\documentclass[a4paper,10pt]{scrartcl}

\usepackage[english] {babel}
\usepackage{amsfonts} 
\usepackage{amsthm} 
\usepackage{amsmath} 
\usepackage[OT1]{fontenc}
\usepackage{a4wide}
\usepackage{amssymb}
\usepackage{mathrsfs}
\usepackage{exscale}
\usepackage{booktabs}
\usepackage{multirow}
\usepackage[numbers,sort]{natbib}
\usepackage{pgfplots}
\usepackage{subfigure}
\usepackage{xcolor}
\usepackage{tikz}
\pgfplotsset{compat=newest}
\pgfplotsset{plot coordinates/math parser=false}

\newlength\fheight
\newlength\fwidth

\newcommand{\Prob}[2]{\mathbb P_{#1} \left(#2\right)}
\newcommand{\E}[2]{\mathbb E_{#1}\left[#2\right]}
\newcommand{\Var}[1]{\mathbb{V} \left[#1\right]}

\DeclareMathOperator*{\Tr}{Tr}

\newcommand{\eps}{\varepsilon}
\newcommand{\1}{{\mathbf 1}}
\newcommand{\calS}{\mathcal{S}}

\begin{document}

\newtheorem{theorem}{Theorem}[section]
\newtheorem{remark}[theorem]{Remark}
\newtheorem{corollary}[theorem]{Corollary}
\newtheorem{lemma}[theorem]{Lemma}
\newtheorem{ex}[theorem]{Example}
\newtheorem{ass}{Assumption}

\begin{center}
\begin{minipage}{.8\textwidth}
\centering 
\LARGE Bump detection in heterogeneous Gaussian regression\\[0.5cm]

\normalsize
\textsc{Farida Enikeeva}\\[0.1cm]
\verb+farida.enikeeva@math.univ-poitiers.fr+\\
Laboratoire de Math\'ematiques et Applications, Universit\'e de Poitiers, France \\
and \\
Institute for Information Transmission Problems, Russian Academy of Science, Moscow, Russia\\[0.1cm]

\textsc{Axel Munk, Frank Werner}\footnotemark[1]\\[0.1cm]
\verb+munk@math.uni-goettingen.de, frank.werner@mpibpc.mpg.de+\\
Felix Bernstein Institute for Mathematical Statistics in the Bioscience, University of G\"ottingen\\
and\\
and Max Planck Institute for Biophysical Chemistry, G\"ottingen, Germany

\end{minipage}
\end{center}

\footnotetext[1]{Corresponding author}

\begin{abstract}

We analyze the effect of a heterogeneous variance on bump detection in a Gaussian regression model. To this end we allow for a simultaneous bump in the variance and specify its impact on the difficulty to detect the null signal against a single bump with known signal strength. This is done by calculating lower and upper bounds, both based on the likelihood ratio. 

Lower and upper bounds together lead to explicit characterizations of the detection boundary in several subregimes depending on the asymptotic behavior of the bump heights in mean and variance. In particular, we explicitly identify those regimes, where the additional information about a simultaneous bump in variance eases the detection problem for the signal. This effect is made explicit in the constant and / or the rate, appearing in the detection boundary.

We also discuss the case of an unknown bump height and provide an adaptive test and some upper bounds in that case.

\end{abstract}

{\it Keywords:}  minimax testing theory, heterogeneous Gaussian regression, change point detection. \\[0.1cm]

{\it AMS classification numbers: } Primary 62G08, 62G10, Secondary 60G15, 62H15. \\[0.3cm]

\section{Introduction}\label{sec:intro}

Assume that we observe random variables $Y = \left(Y_1,...,Y_n\right)$ through the regression model 
\begin{equation}\label{eq:model}
Y_i = \mu_n \left(\frac{i}{n}\right) + \lambda_n\left(\frac{i}{n}\right)Z_i, \qquad i = 1,...,n, 
\end{equation}
where $Z_i \stackrel{\text{i.i.d}}{\sim} \mathcal N \left(0,1\right)$ are observational errors, $\mu_n$ denotes the mean function and $\lambda_n^2$ is the variance function sampled at $n$ equidistant points $i/n$, say.  The aim of this paper is to analyze the effect of a simultaneous change in the variance on the difficulty to detect $\mu_n \equiv 0$ against a mean of the form
\begin{equation}\label{eq:bump}
\mu_n\left(x\right) = \Delta_n 1_{I_n} \left(x\right) = \begin{cases} \Delta_n & \text{if } x \in I_n, \\ 0 & \text{otherwise},\end{cases}
\end{equation}
i.e. a bump with height $\Delta_n>0$ on an interval $I_n \subset\left[0,1\right]$ as location. 
%We will not assume that the location $I_n$ of the bump is known, but we assume that its width $\left|I_n\right|$ is known. For the moment, let us also assume that $\Delta_n$ is known, this will be relaxed later on. 
The assumption that $\Delta_n>0$ is posed only for simplicity here, if $\Delta_n$ is negative or $\left|\mu_n\right|$ is considered, similar results can be obtained analogously. Further, in model \eqref{eq:model} we assume a simultaneous change of the variance by
\begin{equation}\label{eq:model_var}
\lambda_n^2\left(x\right) = \sigma_0^2 + \sigma_n^2 1_{I_n}\left(x\right),\quad x \in \left[0,1\right]
\end{equation}
with a ``baseline'' variance $\sigma_0^2>0$. Note that the additional change in the variance $\sigma_n^2>0$ may only occur when the mean has a bump of size $\Delta_n$ on $I_n$. We assume throughout this paper that $\sigma_n^2$ and $\sigma_0^2$ are known, and that also the length $\left|I_n\right|$ of $I_n$ is known. Note that as the position of $I_n$ is unknown, both functions $\mu_n$ and $\lambda_n$ are (partially) unknown as well. Concerning $\Delta_n$ we will later distinguish between the cases that $\Delta_n$ is known (to which we will provide sharp results) and unknown.

The model arising from \eqref{eq:model}--\eqref{eq:model_var} with the above described parameters will be called the \textbf{heterogeneous bump regression (HBR)} model, which is illustrated in Figure \ref{fig:model}. The aim of this paper is to provide some first insights for precise understanding of the effect of the heterogeneity on detection of the signal bump.

\begin{figure}[!htb]
\setlength\fheight{3cm} \setlength\fwidth{12cm}
\centering
\input{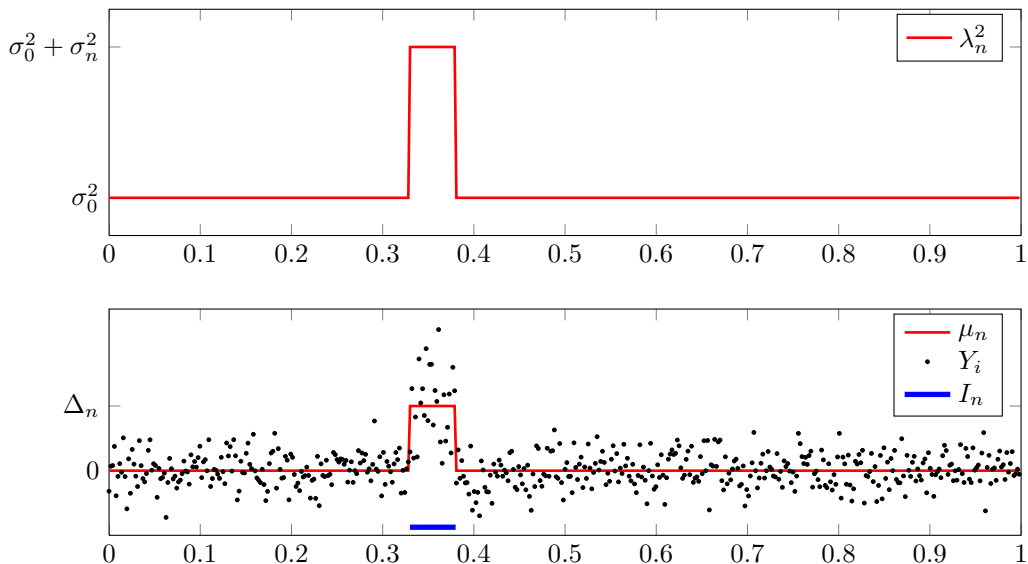}
\caption{The HBR model: Data together with parameters $I_n, \Delta_n, \sigma_0$ and $\sigma_n$ in \eqref{eq:model}--\eqref{eq:model_var} from the HBR. Here $\Delta_n = 4$, $\sigma_0^2 = 1$, $\sigma_n^2 = 4$ and $n = 512$.}
\label{fig:model}
\end{figure}

\subsection{Applications}

The HBR model may serve as a prototypical simplification of more complex situations, where many bumps occur and have to be identified. In fact, in many practical situations it is to be expected that the variance changes when a change of the mean is present; see e.g. \citep{ma10} for a discussion of this in the context of CGH array analysis. Another example arises in econometrics and financial analysis, where it has been argued that price and volatility do jump together \citep{bp03,jt10}. The HBR model can also be viewed as a heterogeneous extension of the ``needle in a haystack'' problem for $\left|I_n\right| \searrow 0$, i.e. to identify s small cluster of non-zero components in a high-dimensional multivariate vector \citep{castro2008finding,dj04, achz08, Baraud:2002, Ingster:2002b}.

\smallskip

Finally we mention the detection of opening and closing states in ion channel recordings (see \citep{ns95} and the references therein), when the so-called open channel noise is present which arises for large channels, e.g. porins \citep{s85,s98}. A small segment of a current recording of a porin in planar lipid bilayers performed in the Steinem lab (Institute of Organic and Biomolecular Chemistry University of G\"ottingen) is depicted in Figure \ref{fig:ion}. We find that the observations fluctuate around two main levels, which correspond to the open and closed state. It is also apparent that the variance changes together with the mean.

\begin{figure}[!htb]
\setlength\fheight{3cm} \setlength\fwidth{12cm}
\centering
\input{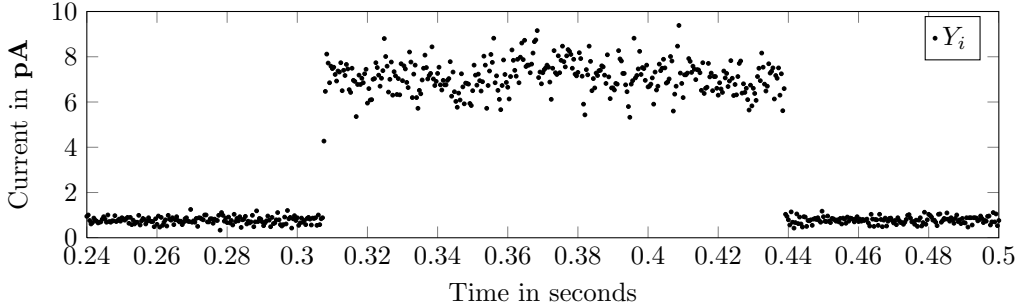}
\caption{Current recordings of a porin in picoampere (\textbf{pA}) sampled every 0.4 milliseconds.}
\label{fig:ion}
\end{figure}

\subsection{Detailed setting and scope of this paper}

Let us now formalize the considered testing problem. Let $I_n \subset \left[0,1\right]$, $n \in\mathbb N$ a given sequence of intervals. W.l.o.g. assume in the following that $l_n := 1/\left|I_n\right|$ only takes values in the integers, i.e. $l_n\in \mathbb N$ for all $n \in \mathbb N$ and define
\begin{equation}\label{eq:defi_A}
\mathcal A_n := \left\{\left[\left(j-1\right) \left|I_n\right|, j \left|I_n\right|\right]~\big|~ 1 \leq j \leq l_n\right\}.
\end{equation}
Then for fixed $n \in \mathbb N$ we want to test
\begin{equation}\label{eq:test_prob}
\begin{aligned}
H_0& : \quad \mu_n\equiv 0, \lambda_n \equiv \sigma_0, \quad \sigma_0>0 \text{ fixed} \\
&\qquad\qquad\text{against}\\
H_1^n& : \exists\ I_n \in \mathcal A_n \text{ s.t. } \mu_n = \Delta_n 1_{I_n}, \ \lambda_n^2 = \sigma_0^2 + \sigma_n^2 1_{I_n}, \quad \sigma_n>0.
\end{aligned}
\end{equation}

Throughout the most of this paper we assume that the parameters $\Delta_n, \sigma_0, \sigma_n$ and $\left|I_n\right|$ are known, only the position $I_n$ of the bump is unknown. Note, that if the location $I_n$ of the bump was also known (and not only its width $\left|I_n\right|$), then the whole problem reduces to a hypothesis test with simple alternative, i.e. the Neyman-Pearson test will be optimal. This situation may be viewed as a variation of the celebrated Behrens-Fisher problem (see e.g. \citep{lr05}) to the situation when the variance changes with a mean change.

The requirement that $\Delta_n$ is known will be relaxed later on.

\smallskip

The main contribution of this paper is to precisely determine the improvement coming from the additional change in variance by determining the detection boundary for \eqref{eq:test_prob}. We stress that this issue has not been addressed before to the authors best knowledge, and it has been completely unclear if and to what extend the additional change in variance adds information to the model so far. We will provide some partial answers to this.

The testing problem \eqref{eq:test_prob} is in fact a simplification of the practically more relevant problem where the interval $I_n$ in the alternative is not necessarily an element of $\mathcal A_n$, but can be arbitrary in $\left[0,1\right]$ with the predefined length $\left|I_n\right|$. Furthermore the parameters $\left|I_n\right|$, $\sigma_0$ and $\sigma_n$ will also be unknown in most practical situations, although often for the in Section 1.1 mentioned channel recordings the variances and the conductance level $\Delta_n$ can be pre-estimated with high accuracy, and hence regarded as practically known. In fact, lack of knowledge of all these quantities leads to a much more complex problem and raises already a remarkable challenge for designing an efficient test, which is far above the scope of this paper. Instead, we aim for analyzing the amount of information contained in the most simple HBR model compared to the homogeneous case, and therefore we restrict ourselves to the simplified (but still quite complicated) situation \eqref{eq:test_prob}. Note that our lower bounds for the detection boundary are also lower bounds for the general situation (although we speculate that then not sharp).

\smallskip

Our results are achieved on the one hand by modifying a technique from D\"umbgen \& Spokoiny \citep{ds01} providing lower bounds in the homogeneous model. We will generalize this to the case of a non-central chi-squared likelihood ratio which appears due to the change in variance. On the other hand we will analyze the likelihood ratio test for \eqref{eq:test_prob} which then provides upper bounds. 
Doing so we use a generalized deviation inequality for the weighted sum of non-central chi-squared random variables from \citep{llm12} (see Appendix \ref{app:tails} for a discussion of this and related inequalities). The authors there also study a heterogeneous regression problem, but in a setting which is not related to bump detection.

\subsection{Analytical tools and literature review}

\paragraph{Minimax testing theory.} In this paper, the analysis of the HBR model is addressed in the context of minimax testing theory, and we refer e.g. to a series of papers by Ingster \citep{i93} and to Tsybakov \citep{t09} for the general methodology. Following this paradigm we aim for determining the detection boundary which marks the borderline between asymptotically detectable and undetectable signals of type \eqref{eq:bump}. 

In the following we consider tests $\Phi_n : \mathbb R^n \to \left\{0,1\right\}$, where $\Phi_n\left(Y\right) = 0$ means that we \textbf{accept} the hypothesis $H_0$ and $\Phi_n \left(Y\right) = 1$ means that we \textbf{reject} the hypothesis. We denote the probability of the type I error by
\[
\bar\alpha \left(\Phi_n, \sigma_0\right) := \E{H_0}{\Phi_n\left(Y\right)} = \Prob{H_0}{\Phi_n\left(Y\right) =1}
\]
and say that the sequence of tests $\Phi_n$ has asymptotic level $\alpha \in \left(0,1\right)$ under $H_0$, if $\limsup\limits_{n \to \infty} \bar\alpha \left(\Phi_n, \sigma_0\right)\leq \alpha$. The type II error is denoted by 
\[
\bar\beta \left(\Phi_n, \Delta_n,\left|I_n\right|,\sigma_0,\sigma_n\right) := \sup_{I_n \in \mathcal A_n} \Prob{\mu_n = \Delta_n 1_{I_n}, \ \lambda_n^2 = \sigma_0^2 + \sigma_n^2 1_{I_n}}{\Phi_n\left(Y\right) = 0}
\]
and $\limsup\limits_{n \to \infty}\bar\beta \left(\Phi_n, \Delta_n,\left|I_n\right|,\sigma_0,\sigma_n\right)$ is the asymptotic type II error of the sequence of tests $\Phi_n$ under hypothesis $H_0$ and alternative $H_1^n$ determined by the parameters $\Delta_n, \left|I_n\right|, \sigma_0$ and $\sigma_n$. Below we will skip the dependence on the parameters to ease the presentation.

\smallskip

Obviously, a fixed bump can be detected always with the sum of type I and II errors tending to $0$ as $n \to \infty$. Thus we are interested in asymptotically vanishing bumps, hence we want to investigate which asymptotic behavior of $\Delta_n$, $\sigma_n$ and $\left|I_n\right|\searrow 0$ implies the following twofold condition for a fixed $\alpha>0$ (cf. \cite{kk11}):
\begin{itemize}
\item[a)]\textit{lower detection bound:} let $\tilde H_1^n$ consist of faster vanishing bumps. Then for any test with $\limsup\limits_{n \to \infty} \bar\alpha \left(\Phi_n\right)\leq \alpha$ it holds $\liminf\limits_{n \to \infty} \bar\beta \left(\Phi_n\right) \geq 1-\alpha$.
\item[b)]\textit{upper detection bound:} let $\tilde H_1^n$ consist of slower vanishing bumps. Then there is a sequence of tests with $\limsup\limits_{n \to \infty} \bar\alpha \left(\Phi_n\right)\leq \alpha$ and $\limsup\limits_{n \to \infty} \bar\beta \left(\Phi_n\right) \leq \alpha$.
\end{itemize}
Typically the terms faster and slower are measured by changing the constant appearing in the detection boundary by an $\pm\eps_n$ term. If upper and lower bounds coincide (asymptotically), we speak of the (asymptotic) detection boundary, cf. \eqref{eq:db_known} below.

In fact we will see that for our specific situation the obtained asymptotic behavior $\Delta_n$, $\sigma_n$ and $\left|I_n\right|\searrow 0$ does not depend on the specific choice of $\alpha \in \left(0,1\right)$. Consequently, we determine conditions for $H_0$ and $H_1^n$ to be either \textit{consistently distinguishable} or \textit{asymptotically indistinguishable} in the spirit of Ingster (cf. \citep{i93}).

\paragraph{Literature review.} Motivated by examples from biology, medicine, economics and other applied sciences, methods for detecting changes or bumps in a sequence of homogeneous observations have been studied extensively in statistics and related fields. We refer to Cs\"orgo \& Horv\'{a}th \citep{ch97}, Siegmund \citep{s13}, Carlstein et al. \citep{cms94} or Frick et al. \citep{fms14} for a survey. Estimating a piecewise constant function in this setting consists on the one hand in a model selection problem (namely estimating the number of change points) and on the other hand in a classical estimation problem (estimating the locations and amplitudes). For the former, we refer to Yao \citep{y88}, who proposed a method to estimate the total number of change points in the data by Schwarz' criterion, and to Bai \& Perron \citep{bp98}, who investigate a Wald-type test. For a more general approach to treat the change-point problem from a model selection perspective we refer to Birg\'e \& Massart \citep{bm01}, and for a combined approach with confidence statements and estimation, see Frick, Munk \& Sieling \cite{fms14}. Concerning the latter, Siegmund \& Venkatraman \citep{sv95} used a generalized likelihood ratio to estimate the location of a single change point in the data. The overall estimation problem has e.g. been tackled by Boysen et al. \cite{b09} via $\ell_0$-penalized least-squares and by Harchaoui \& L\'evy-Leduc \citep{hl10} via a TV-penalized least-squares approach. Siegmund, Yakir \& Zhang \citep{syz11} proposed a specialized estimator for the situation of many aligned sequences when a change occurs in a given fraction of those. Recent developments include efficient algorithms for the overall estimation problem based on accelerated dynamic programming and optimal partitioning by Killick, Fearnhead \& Eckley \citep{kfe12} and \cite{fms14}, on penalized or averaged log likelihood ratio statistics by Rivera \& Walther \citep{rw13}, or Bayesian techniques as proposed by Du, Kao \& Kou \citep{dkk15}. We also mention the recent paper by Goldenshluger, Juditski \& Nemirovski \citep{gjn15}, which has been discussed in this context \citep{mw15}. 

The special question of minimax testing in such a setting has also been addressed by several authors. We mention the seminal paper by D\"umbgen \& Spokoiny \citep{ds01} and also D\"umbgen \& Walther \citep{dw08} who studied multiscale testing in a more general setup, and \cite{fms14} for jump detection in change point regression. Furthermore, we refer to Chen \& Walther \citep{cw13} for an optimality consideration of the tests introduced in \citep{rw13}, and to a rigorous minimax approach for detecting sparse jump locations by Jeng, Cai \& Li \citep{jcl10}. Nevertheless, all these papers only address the case of a homogeneous variance, i.e. $\lambda_n \equiv \sigma_0$. There, it is well-known that the detection boundary is determined by the equations
\begin{equation}\label{eq:db_known}
\sqrt{n \left|I_n\right|} \Delta_n  = \sqrt{2}\sigma_0\sqrt{-\log\left(\left|I_n\right|\right)},
\end{equation}
and the terms ``faster'' and ``slower'' in the above definition are expressed by replacing $\sqrt{2}\sigma_0$ by $\sqrt{2}\sigma_0 \pm \eps_n$ with a sequence $\left(\eps_n\right)_{n \in \mathbb N}$ such that $\eps_n \to 0, \eps_n \sqrt{-\log\left(\left|I_n\right|\right)} \to \infty$ as $n \to \infty$ (for details see \citep{fms14}). 

\smallskip

In the heterogeneous model \eqref{eq:model}--\eqref{eq:model_var} and the connected testing problem \eqref{eq:test_prob} mainly two situations have to be distinguished, see also \citep{l05}. 

The first situation emerges when it is not known whether the variance changes (but it might). Then as in \eqref{eq:model_var} we explicitly admit $\sigma_n^2 = 0$ and the variance is a nuisance parameter rather than an informative part of the model for identification of the bump in the mean. Hence, no improvement of the bound \eqref{eq:db_known} for the case of homogeneous variance is to be expected. The estimation of the corresponding parameters has also been considered in the literature. Huang \& Chang \citep{hc93} consider a least-squares type estimator, and Braun, Braun \& M\"uller \citep{bbm00} combine quasilikelihood methods with Schwarz' criterion for the model selection step. We also mention Arlot \& Celisse \citep{ac11} for an approach based on cross-valiation and Boutahar \citep{b12} for a CUSUM-type test in a slightly different setup. Even though, a rigorous minimax testing theory remains elusive and will be postponed to further investigations.

The second situation, which we treat in this paper, is different: We will always assume $\sigma_n^2>0$ (as in the examples in Section 1.1), which potentially provides extra information on the \textit{location} of the bump and detection of a change might be improved. This is \textit{not} the case in the situation discussed above, as the possibility of a change in variance can only complicate the testing procedure, whereas $\sigma_n^2>0$ gives rise to a second source of information (the variance) for signal detection. The central question is: Does the detection boundary improve due to the additional information and if so, how much? 

\smallskip

\paragraph{Relation to inhomogeneous mixture models.} Obviously, the HBR model is related to a Gaussian mixture model
\begin{equation}\label{eq:mixture}
Y_i \sim \left(1-\varepsilon\right) \mathcal N \left(0,\sigma_0^2\right) + \varepsilon \mathcal N \left(\Delta, \sigma^2\right), \qquad i=1,...,n,
\end{equation}
which has been introduced to analyze heterogeneity in a linear model with a different focus, see e.g. \citep{cw14,aw13,cjj11}. For $\varepsilon \sim n^{-\beta}$ ($0 < \beta < 1$) different (asymptotic) regimes occur. If $\beta > 1/2$, the non-null effects in the model are sparse, which leads to a different behavior as if they are dense ($\beta \leq 1/2$). Although it is not possible to relate \eqref{eq:mixture} with the HBR model in general, it is insightful to highlight some distinctions and commonalities in certain scenarios (for the homogeneous case see the discussion part of \citep{fms14}). A main difference to our model is that the non-null effects do not have to be clustered on an interval $I_n$. It is exactly this clustering, which provides the additional amount of information due to the variance heterogeneity in the HBR model. A further difference is that the definition of an i.i.d. mixture as in \eqref{eq:mixture} intrinsically relates variance and expectation of the $Y_i$. 

\subsection{Organization of the paper} 

The paper is organized as follows: In the following Section \ref{sec:overview} we present our findings of the detection boundary. The methodology and the corresponding general results are stated in Section \ref{sec:method}. Afterwards we derive lower bounds in Section \ref{sec:lower_bounds} and upper bounds in Section \ref{sec:upper_bounds}. There we also discuss upper bounds for the case that $\Delta_n$ is unknown and provide a likelihood ratio test which adapts to this situation. In Section \ref{sec:simulations} we present some simulations and compare the finite sample power of the non-adaptive test and the adaptive test. Finally we will discuss some open questions in Section \ref{sec:conclusions}. To ease the presentation all proofs will be postponed to the Appendix.

\section{Overview of results}\label{sec:overview}

Throughout the following we need notations for asymptotic inequalities. For two sequences $\left(a_n\right)_{n \in \mathbb N}$ and $\left(b_n\right)_{n \in \mathbb N}$, we say that $a_n$ is asymptotically less or equal to $b_n$ and write $a_n \precsim b_n$ if there exists $N \in \mathbb N$ such that $a_n \leq b_n$ for all $n \geq N$. Similarly we define $a_n \succsim b_n$. If $a_n / b_n \to c$ as $n \to \infty$ for some $c \in \mathbb R \setminus \left\{0\right\}$ we write $a_n \sim b_n$. If $c = 1$ we write $a_n \asymp b_n$. We use the terminology to say that $a_n$ and $b_n$ have the same asymptotic rate (but probably different constant) if $a_n \sim b_n$, and that they have the same rate \textbf{and} constant if $a_n \asymp b_n$. Consequently, the relation in \eqref{eq:db_known} determining the detection boundary in the homogeneous case becomes
\[
\sqrt{n \left|I_n\right|} \Delta_n  \asymp \sqrt{2}\sigma_0\sqrt{-\log\left(\left|I_n\right|\right)}.
\]

The notations $\sim$ and $\asymp$ coincide with standard notations. The definitions of $\precsim$ and $\succsim$ can be seen as extensions of the classical minimax notations by Ingster \citep{i93} or D\"umbgen \& Spokoiny \citep{ds01}, where the null hypothesis is tested against the complement of a ball around $0$ with varying radius $\rho = \rho\left(n\right)$. The results are typically presented by stating that if $\rho\left(n\right)$ tends to $0$ faster than the boundary function $\rho^* \left(n\right)$, or if their ratio converges to some constant $<1$, then the null and the alternative hypotheses are indistinguishable. 

In our setting, the asymptotic conditions are defined in terms of sums rather than ratios. Thus we introduced $\precsim$ and $\succsim$ to have similar notations.

\smallskip

In the HBR model, our signal has a change in mean determined by $\Delta_n$ and also a change in its variance, which will be described by the parameter
\[
\kappa_n := \frac{\sigma_{n}}{\sigma_0}>0.
\]
In fact, we will see that the detection boundary is effectively determined by the ratio of $\kappa_n^2$ and $\Delta_n$, which leads to three different regimes:
\begin{itemize}
\item[{\bf DMR}] \textbf{Dominant mean regime.} If $\kappa_n^2$ vanishes faster than $\Delta_n$, the mean will asymptotically dominate the testing problem. In this case, the additional information will asymptotically vanish too fast, so that we expect the same undetectable signals as for the homogeneous case. The dominant mean regime consists of all cases in which $\kappa_n^2 / \Delta_n \to 0$, $n \to \infty$. 
\item[{\bf ER}] \textbf{Equilibrium regime.} If $\kappa_n^2$ has a similar asymptotic behavior as $\Delta_n$, we expect a gain from the additional information. The equilibrium regime consists of all cases in which $\kappa_n \to 0$, $\Delta_n \to 0$ and $c := \lim_{n \to \infty} \kappa_n^2 / (\Delta_n/\sigma_0) = \lim_{n \to \infty} \sigma_n^2/(\Delta_n \sigma_0)$ satisfies $0 < c < \infty$. 
\item[{\bf DVR}] \textbf{Dominant variance regime.} If $\Delta_n$ vanishes faster than $\kappa_n^2$, the variance will asymptotically dominate the testing problem. In this case, we expect the same detection boundary as for the case of testing for a jump in variance only. The dominant variance regime consists of all cases in which $\kappa_n \to 0$ and $\kappa_n^2 / \Delta_n \to \infty$, $n \to \infty$.
\end{itemize}

\bigskip

With this notation, we can collect the most important results of this paper in the following Table \ref{tab:results}. The corresponding detection boundaries in different regimes are also illustrated in Figure \ref{fig:constants}.

\begin{table}[!htb]
\caption{Main results of the paper.}
\begin{tabular}{|c|c|c|c|c|}
\toprule[1pt]
 & rate & \multicolumn{3}{|c|}{constant} \\[0.1cm]
\cmidrule(lr){3-5}
& & lower bound & \multicolumn{2}{c|}{upper bound}\\[0.1cm]
\cmidrule(lr){4-5}
& & & $\Delta_n$ known & $\Delta_n$ unknown \\[0.1cm]
\midrule[.8pt]
\multirow{2}{*}{\vspace*{-.2cm}DMR} & \multirow{2}{*}{\vspace*{-.2cm}$\sqrt{n \left|I_n\right|}\Delta_n \sim \sqrt{-\log\left(\left|I_n\right|\right)}$} & $ \sqrt{2}\sigma_0-\eps_n$ & $\sqrt{2}\sigma_0+\eps_n $ & $ \sqrt{2}\sigma_0+\eps_n$ \\[0.1cm]
& & Thm. \ref{thm:db_signal} & Thm. \ref{thm:LR_homogeneous} & Thm. \ref{thm:LR_adaptive} \\[0.1cm]
\midrule[.5pt]
\multirow{4}{*}{\vspace*{-.8cm}ER} & \multirow{2}{*}{\vspace*{-.2cm}$\sqrt{n \left|I_n\right|}\Delta_n \sim \sqrt{-\log\left(\left|I_n\right|\right)}$} & $ \sqrt{2}\sigma_0 \sqrt{\frac{2}{2+c^2}}-\eps_n$ & $ \sqrt{2}\sigma_0\sqrt{\frac{2}{2+c^2}}+\eps_n $ & $ \sigma_0 \frac{c+\sqrt{2+3c^2}}{1+c^2}+\eps_n $ \\[0.1cm]
& & Thm. \ref{thm:db_equilibrium} & Thm. \ref{thm:LR_equilibrium} & Thm. \ref{thm:LR_adaptive} \\[0.1cm]
\cmidrule[.8pt]{2-5}
 & \multirow{2}{*}{\vspace*{-.2cm}$\sqrt{n \left|I_n\right|}\kappa^2_n \sim \sqrt{-\log\left(\left|I_n\right|\right)}$} & $2 \sqrt{\frac{c^2}{2 + c^2}}-\eps_n$ & $2 \sqrt{\frac{c^2}{2 + c^2}}+\eps_n $ & $ c \frac{c+\sqrt{2+3c^2}}{1+c^2}+\eps_n $ \\[0.1cm]
& & cf. \eqref{eq:lower_bound_equilibrium1_equiv} & analog to \eqref{eq:lower_bound_equilibrium1_equiv} & analog to \eqref{eq:lower_bound_equilibrium1_equiv} \\[0.1cm]
\midrule[.5pt] 
\multirow{2}{*}{\vspace*{-.2cm}DVR} & \multirow{2}{*}{\vspace*{-.2cm}$\sqrt{n \left|I_n\right|}\kappa_n^2 \sim \sqrt{-\log\left(\left|I_n\right|\right)}$} & $2-\eps_n$ & $ 2+\eps_n $ & $ 1+\sqrt{3}+\eps_n $\\[0.1cm]
& & Thm. \ref{thm:db_variance} & Thm. \ref{thm:LR_dominantvar} & Thm. \ref{thm:LR_adaptive} \\[0.1cm]
\bottomrule[1pt]
\end{tabular}
 
\smallskip
 
\parbox{\textwidth}{Here $\left(\eps_n\right)$ is any sequence such that $\eps_n \to 0, \eps_n \sqrt{-\log\left(\left|I_n\right|\right)} \to \infty$. \\

The second column depicts the rates obtained in the different regimes, the columns three to five give the constants in different situations. ER is stated twice, because the results are shown w.r.t. the different rates from DMR and DVR respectively. \\

Exemplary, the lower bound entry in the DMR denotes that signals are no longer detectable if $\sqrt{n \left|I_n\right|}\Delta_n \precsim \left(\sqrt{2 }\sigma_0 - \eps_n\right) \sqrt{-\log\left(\left|I_n\right|\right)}$. Vice versa, the upper bound entry in the DMR means that signals are detectable as soon as $\sqrt{n \left|I_n\right|}\Delta_n \succsim \left(\sqrt{2 }\sigma_0 + \eps_n\right) \sqrt{-\log\left(\left|I_n\right|\right)}$, no matter if $\Delta_n$ is known or needs to be estimated from the data.}

\label{tab:results}
\end{table}

\paragraph{$\Delta_n$ known.} It can readily be seen from Table \ref{tab:results} that the lower and the upper bounds with known $\Delta_n$ coincide up to the $\pm \eps_n$ term in all regimes. This directly implies that the detection boundaries are determined by these constants, which we will describe in more detail now.

\smallskip

{\bf DMR:} A comparison of the lower and upper bounds in Table \ref{tab:results} yields that the detection boundary is given by
\begin{equation}\label{eq:db_signal}
\sqrt{n \left|I_n\right|} \Delta_n  \asymp \sqrt{2}\sigma_0\sqrt{-\log\left(\left|I_n\right|\right)}.
\end{equation}
Hence, the detection boundary in the dominant mean regime coincides with the detection boundary in the homogeneous model (cf. \eqref{eq:db_known}). More precisely, if the additional information $\kappa_n^2$ about a jump in the variance vanishes faster than the jump $\Delta_n$ in mean, we cannot profit from this information asymptotically.

\smallskip

{\bf ER:} It follows similarly that the detection boundary is given by
\begin{equation}\label{eq:db_equilibrium}
\sqrt{n \left|I_n\right|} \Delta_n  \asymp \sqrt{2}\sigma_0 \sqrt{\frac{2}{2+c^2}}\sqrt{-\log\left(\left|I_n\right|\right)},
\end{equation}
where $c = \sigma_0^{-1} \lim_{n \to \infty} \sigma_n^2 / \Delta_n$. The characterization \eqref{eq:db_equilibrium} shows that we do always profit from the additional information in the ER case as $\sqrt{\frac{2}{2+c^2}}<1$ whenever $c>0$. Compared to the homogeneous model or the dominant mean regime the constant improves but the rate stays the same. Note, that the improvement can be quite substantial, e.g. if $c = 1$, which amounts to the same magnitude of $\sigma_n$ and $\Delta_n$, the additional bump in the variance leads to a reduction of $33\%$ sample size to achieve the same power compared to homogeneous bump detection. If $c = 2$, we obtain a reduction of $66\%$.

As $\Delta_n$ and $\kappa_n^2$ are of the same order in the ER, we may replace $\Delta_n$ by $\kappa_n^2$ in \eqref{eq:db_equilibrium} and obtain the equivalent characterization
\begin{equation}\label{eq:db_equilibrium_equiv}
\sqrt{n \left|I_n\right|} \kappa_n^2  \asymp 2 \sqrt{\frac{c^2}{2 + c^2}}\sqrt{-\log\left(\left|I_n\right|\right)}.
\end{equation}
This formulation allows for a comparison with the DVR below and is also depicted in Table \ref{tab:results}, second ER row.

\smallskip

{\bf DVR:} We find from entries three and four in the DVR row of Table \ref{tab:results} that the detection boundary is given by
\begin{equation}\label{eq:db_variance}
\sqrt{n \left|I_n\right|} \kappa_n^2  \asymp 2\sqrt{-\log\left(\left|I_n\right|\right)}.
\end{equation}
If the jump in variance asymptotically dominates the jump in mean, we exactly obtain the detection boundary for a jump in variance only. This is a natural analog to \eqref{eq:db_signal} and coincides with the findings in the literature, cf. \citep[Sect. 2.8.7.]{ch97}. 

Finally note that if $c$ in \eqref{eq:db_equilibrium_equiv} tends to $\infty$, we end up with \eqref{eq:db_variance}. In this spirit a constant $c < \infty$ can be also seen as an improvement over the pure DVR where the rate stays the same but the constant decreases, which is an analog to the comparison of the DMR and the ER, see above. A summary of the obtained detection boundary statements is illustrated in Figure \ref{fig:constants}.

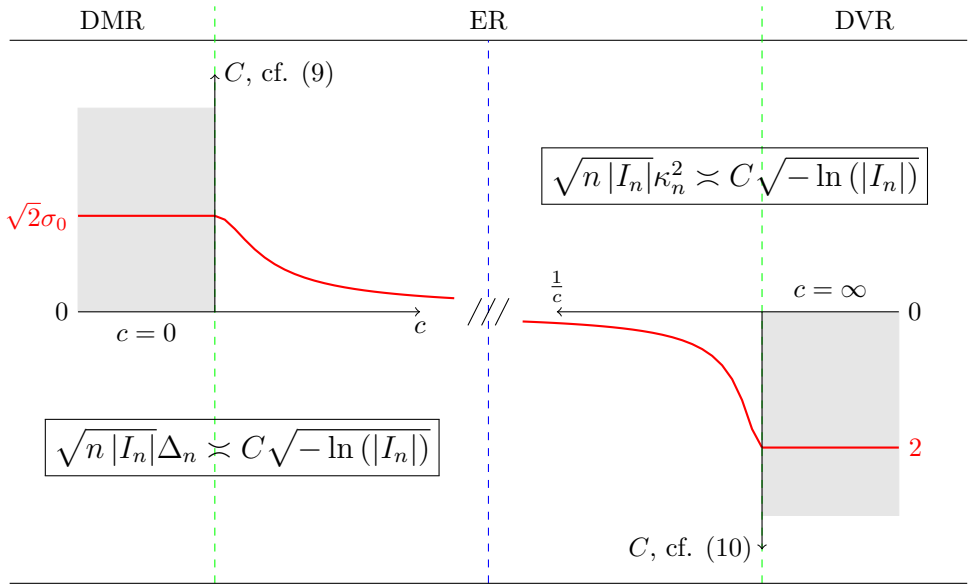
\begin{figure}[!htb]
\begin{tikzpicture}[scale=.9]
\draw (1.5,8.3) node {DMR};
\draw (7,8.3) node {ER};
\draw (12.5,8.3) node {DVR};

\draw (0,8) -- (14,8);

\filldraw[fill=gray!20,very thin,draw=gray!20] (1,4) rectangle (3,7);
%\draw[very thin,color=gray] (1,4) grid (5.5,7);

\filldraw[fill=gray!20,draw=gray!20,very thin] (11,1) rectangle (13,4);
%\draw[very thin,color=gray] (8.5,1) grid (13,4);

\draw[green,dashed] (3,0) -- (3,8.5);
\draw[green,dashed] (11,0) -- (11,8.5);
\draw[blue,dashed] (7,0) -- (7,8);

\draw (6.9,3.8) -- (7.1,4.2);
\draw (6.7,3.8) -- (6.9,4.2);
\draw (7.1,3.8) -- (7.3,4.2);

\draw[color=red,thick] (3,5.4142) -- (1,5.4142) node[left] {$\sqrt{2}\sigma_0$};
\draw (1,4) node[left] {$0$};
\draw[color=red,thick,domain=3:6.5] plot[id=C1] function{4+sqrt(2)/(sqrt(1+8*(x-3)*(x-3)/2))};

\draw[color=red,thick] (11,2) -- (13,2) node[right] {$2$};
\draw (13,4) node[right] {$0$};
\draw[color=red,thick,domain=7.5:11] plot[id=C2] function{4-2/(sqrt(8*2*(11-x)*(11-x)+1))};

\draw[->] (1,4) -- (6,4) node[below] {$c$};
\draw[->] (3,4) -- (3,7.5) node[right] {$C$, cf. \eqref{eq:db_equilibrium}};
\draw (2,3.7) node {$c = 0$};

\draw[->] (13,4) -- (8,4) node[above] {$\frac{1}{c}$};
\draw[->] (11,4) -- (11,.5) node[left] {$C$, cf. \eqref{eq:db_equilibrium_equiv}};
\draw (12,4.3) node {$c = \infty$};

\draw (0,0) -- (14,0);

\draw (3.4,2) node {\fbox{\large $\sqrt{n \left|I_n\right|} \Delta_n \asymp C \sqrt{-\ln\left(\left|I_n\right|\right)}$}};
\draw (10.6,6) node {\fbox{\large $\sqrt{n \left|I_n\right|} \kappa_n^2 \asymp C \sqrt{-\ln\left(\left|I_n\right|\right)}$}};

\end{tikzpicture}
\caption{Illustration of the detection boundaries \eqref{eq:db_signal}, \eqref{eq:db_equilibrium}, \eqref{eq:db_equilibrium_equiv} and \eqref{eq:db_variance} in the different regimes. The dashed green lines illustrate the phase transitions between the DMR, ER and ER, DVR respectively, whereas the dashed blue line depicts the phase transition between the rates for DVR and DMR. Note, that on the left-hand side the $x$-axis corresponds to $c$, whereas on the right-hand side $x$- and $y$-axis are inverted and the $x$-axis belongs to $1/c$.}
\label{fig:constants}
\end{figure}

\paragraph{$\Delta_n$ unknown.} In case that $\Delta_n$ is unknown, it has to be estimated from the data $Y$ and the corresponding likelihood ratio test will be introduced in Section 4.4. The change in the considered statistics leads to different upper bounds. It is readily seen from Table \ref{tab:results} that the obtained upper bounds do not coincide with the lower bounds in all regimes. 

\smallskip

{\bf DMR:} It follows from the last entry in the DMR row that \eqref{eq:db_signal} also characterizes the adaptive detection boundary for unknown $\Delta_n$.

\smallskip

{\bf ER:} If the parameter $\Delta_n$ is unknown, it can be seen from the last entry in the ER row that it is unclear if the lower bound remains sharp. We do not have an adaptive lower bound which would allow for an explicit statement here.

\smallskip

{\bf DVR:} In this regime it is again unclear if the lower bound is also sharp in case that $\Delta_n$ is unknown (again we do not have an adaptive one), but we consider it likely that the obtained upper bound $1 + \sqrt{3}$ (cf. the last entry in the DVR row) is sharp.

\smallskip

This loss of a constant can be interpreted as the price for adaptation and is quantified by the ratio $r \left(c\right)$  of \eqref{eq:adaptive_condition_homogeneous} and \eqref{eq:condition_homogeneous}, \eqref{eq:adaptive_condition_equilibrium} and \eqref{eq:condition_equilibrium} and \eqref{eq:adaptive_condition_dominantvar} and \eqref{eq:condition_dominantvar} respectively. The ratios are given by
\[
r\left(c\right) = \begin{cases}
1 & \text{DMR}, c = 0,\\
\frac{\sqrt{2+c^2}\left(c+\sqrt{2+3c^2}\right)}{2\left(1+c^2\right)} & \text{ER}, 0 < c < \infty,\\
\frac{1+\sqrt{3}}{2} & \text{DVR}, c = \infty,
\end{cases}
\]
which are displayed in Figure \ref{fig:adaptation}. Note that $r$ is a continuous function. Remarkably, the price for adaptation is never larger then $\sqrt{2}$, as $r$ attains its maximum at $c = \sqrt{2}$, such that $r\left(\sqrt{2}\right) = \sqrt{2}$. As $c \to \infty$, $r$ tends to $(1+\sqrt{3})/2$, as $c \to 0$, it tends to $1$, meaning that adaptation has no cost at all in this situation. Note, that this is in line with findings for the homogeneous Gaussian model \citep{fms14}. 

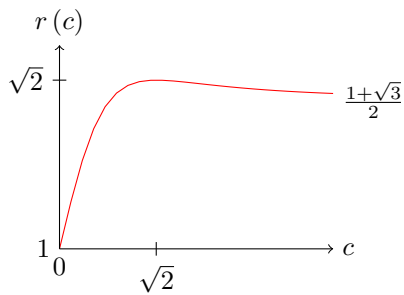
\begin{figure}[!htb]
\centering
\begin{tikzpicture}[scale=.9]
\draw[color=red,domain=0:4] plot[id=r] function{6*((sqrt(2+x*x)*(x+sqrt(2+3*x*x)))/2/(1+x*x)-1)};

\draw[->] (0,0) -- (4,0) node[right] {$c$};
\draw[->] (0,0) -- (0,3) node[above] {$r\left(c\right)$};

\draw (0,0) node[left] {$1$};
\draw (4,2.1962) node[right] {$\frac{1+\sqrt{3}}{2}$};
\draw (.1,2.4853) -- (-.1,2.4853) node[left] {$\sqrt{2}$};
\draw (1.4142,.1) -- (1.4142,-.1) node[below] {$\sqrt{2}$};
\draw (0,0) node[below] {$0$};
\end{tikzpicture}
\caption{The price of adaptation $r$ plotted against the parameter $0 \leq c < \infty$.}
\label{fig:adaptation}
\end{figure}

\bigskip

If we relax the regimes ER and DVR by allowing for $\sigma_n  \to \sigma>0$, things become more complicated. We will prove lower bounds for those cases as well, which include logarithmic terms in $\kappa = \lim_{n \to \infty} \kappa_n$ (cf. Theorems \ref{thm:db_equilibrium} and \ref{thm:db_variance} and values of $C$ in \eqref{eq:lower_bound_equilibrium2} and \eqref{eq:lower_bound_variance2} respectively). Our general methodology (cf. Theorem \ref{thm:LR_gen}) will also allow for upper bounds in these situations (cf. \eqref{eq:er_relaxed_ub} and \eqref{eq:dvr_relaxed_ub}), but it is clear that they cannot be optimal for the following simple reason: The underlying deviation inequality does not include logarithmic terms in $\kappa$, neither do any from the literature (cf. Appendix \ref{app:tails}).

\section{General methodological results}\label{sec:method}

Fixing $I_n$ in \eqref{eq:test_prob}, we obtain the likelihood ratio by straightforward calculations as
\begin{align}
 L_{n}(\Delta_n,& I_n,\kappa_n;Y)=\left(\kappa_n^2+1\right)^{-\frac{n\left|I_n\right|}{2}}  \exp\Biggl(\sum\limits_{i: \frac{i}{n} \in I_n}\frac{\kappa_n^2 Y_i^2 + 2 Y_i \Delta_n - \Delta_n^2}{2\left(1+\kappa_n^2\right)\sigma_0^2}\Biggr) \nonumber\\[0.1cm]
&=\left(\kappa_n^2+1\right)^{-\frac{n\left|I_n\right|}{2}}  \exp\Biggl(\frac{\kappa_n^2}{2\left(1+\kappa_n^2\right)\sigma_0^2} \sum\limits_{i: \frac{i}{n} \in I_n} \left(Y_i + \frac{\Delta_n}{\kappa_n^2} \right)^2 - \frac{n |I_n|\Delta_n^2}{2\sigma_0^2\kappa_n^2}\Biggr).\label{eq:likelihood_ratio}
\end{align}

For the general testing problem \eqref{eq:test_prob} where only the size $\left|I_n\right|$ is known but the location $I_n$ not, the likelihood ratio function is then given by $\sup_{I_n \in \mathcal A_n}L_{n}\left(\Delta_n,I_n,\kappa_n;Y\right)$. 

\subsection{Lower bounds}

Let $\Phi_n$ be any test with asymptotic significance level $\alpha$ under the null-hypothesis $\mu \equiv 0, \lambda_n \equiv \sigma_0$. We will now determine classes $\calS := \left(S_n\right)_{n \in \mathbb N}$ of bump functions, such that $\Phi_n$ is not able to differentiate between the null hypothesis $\mu \equiv 0, \lambda_n \equiv \sigma_0$ and functions $\left(\mu_n,\lambda_n\right) \in S_n$ with type II error $\leq 1-\alpha$. To this end, we construct a sequence $\calS$ such that
\begin{equation}\label{eq:statement}
\bar\alpha \left(\Phi_n\right) \leq \alpha + o \left(1\right) \qquad\Rightarrow\qquad 1-\bar\beta \left(\Phi_n\right)- \alpha \leq o\left(1\right),
\end{equation}
which is equivalent. For such classes $\calS$ we say that $\calS$ is undetectable. 

In the following, the sequence $\calS$ will always be characterized by asymptotic requirements on $\Delta_n, \kappa_n$ and $\left|I_n\right|$. To keep the notation as simple as possible, only the asymptotic requirements will be stated below, meaning that the sequence $\calS$
\[
S_n = \left\{\mu_n = \Delta_n 1_{I_n}, \lambda_n^2 = \sigma_0^2 + \sigma_n^2 1_{I_n} ~\big|~ \Delta_n, \left|I_n\right|, \kappa_n \text{ satisfy the specified requirements}\right\} 
\]
is undetectable.

To prove lower bounds we use the following estimate, which has been employed in \citep{ds01, fms14} as well:
\begin{lemma}\label{lem:estimate_1}
Assume that \eqref{eq:model}--\eqref{eq:model_var} hold true and $\mathcal A_n$ is given by \eqref{eq:defi_A}. If $\Phi_n$ is a test with asymptotic level $\alpha\in \left(0,1\right)$ under the null-hypothesis $\mu \equiv 0$, i.e. $\bar\alpha \left(\Phi_n\right) \leq \alpha + o \left(1\right)$, then
\begin{equation}\label{eq:estimate_1}
1-\bar\beta \left(\Phi_n\right)- \alpha \leq \E{\mu \equiv 0,\lambda\equiv\sigma_0}{\left| \frac{1}{l_n} \sum\limits_{I_n \in \mathcal A_n} L_{n}\left(\Delta_n,I_n,\kappa_n,Y\right)-1\right|} + o \left(1\right),
\end{equation}
as $n \to \infty$.
\end{lemma}
Recall that we want to construct $\calS$ such that the right-hand side of \eqref{eq:estimate_1} tends to $0$. This can be achieved (as in \citep[Lemma 6.2]{ds01}) by the weak law of large numbers. By controlling the moments of $L_{n}\left(\Delta_n,I_n,\kappa_n,Y\right)$, this leads to the following theorem characterizing undetectable sets $\calS$:
\begin{theorem}\label{thm:characterization}
Assume the HBR model is valid with known parameters $\Delta_n, \kappa_n$ and $\left|I_n\right|$. The sequence $\calS$ determining the asymptotic behavior of $\Delta_n$, $\kappa_n$ and $\left|I_n\right|$ is undetectable, if $\left|I_n\right|\searrow 0$ and there exists a sequence $\delta_n>0$, %$p_n \in \left[0,1\right]$ 
satisfying $\delta_n < 1/\kappa_n^2$ such that for $n \to \infty$, 
\begin{equation}\label{eq:cond}
\frac{n \left|I_n\right|\Delta_n^2}{2 \sigma_0^2}\frac{\left(1+\delta_n\right)\delta_n}{1-\delta_n\kappa_n^2}  - \delta_n \frac{n\left|I_n\right|}{2} \log\left(1+\kappa_n^2\right) 
 - \frac{n\left|I_n\right|}{2}  \log\left(1- \delta_n\kappa_n^2\right) + \delta_n \log\left(\left|I_n\right|\right)  \to -\infty.
\end{equation}
\end{theorem}

\subsection{Upper bounds}

To construct upper bounds we will now consider the likelihood ratio test. Recall that $\left|I_n\right|$ is known, but the true location of the bump is unknown. Motivated by \eqref{eq:likelihood_ratio} it is determined by
\begin{equation}\label{eq:T_statistic}
T_n^{\Delta_n,\kappa_n,\left|I_n\right|,\sigma_0} \left(Y\right) := \sup\limits_{I_n \in \mathcal A_n} \frac{1}{\sigma_0^2} \sum\limits_{i: \frac{i}{n} \in I_n} \left(Y_i + \frac{\Delta_n}{\kappa_n^2} \right)^2
\end{equation}
as test statistic. For simplicity of presentation, we will just write $T_n$ for the statistic in \eqref{eq:T_statistic} in the following and drop the dependence on the parameters.

Furthermore let
\begin{equation}\label{eq:LR_test}
\Phi_n\left(Y\right) := \begin{cases} 1 & \text{if } T_n \left(Y\right) > c_{\alpha,n}^*, \\[0.1cm] 0 & \text{else} \end{cases}
\end{equation}
be the corresponding test where $c_{\alpha,n}^* \in \mathbb R$ is determined by the test level. 

In the following we will be able to analyze the likelihood ratio test \eqref{eq:LR_test} in DMR, ER and DVR with the help of Lemma \ref{lem:tails}. For the relaxed situations where $\kappa_n \not\to 0$ we will prove lower bounds including logarithmic terms in Theorems \ref{thm:db_equilibrium} and \ref{thm:db_variance}. Upper bounds including logarithmic terms cannot be obtained by the deviation inequalities in Lemma \ref{lem:tails}, and we furthermore found no deviation inequality including logarithmic terms in the literature. Even worse we are not in position to prove such an inequality here, and thus we will not be able to provide upper bounds which coincide with the lower ones in the relaxed regimes.

\smallskip

Now we are able to present the main theorem of this Section:
\begin{theorem}\label{thm:LR_gen}
Assume the HBR model with $\left|I_n\right|\searrow 0$, $\sigma_n>0$ holds true and let $H_0, H_1^n$ be as in \eqref{eq:test_prob}. Let $\alpha \in \left(0,1\right)$ be a given significance level and
\begin{equation}\label{eq:threshold}
c_{\alpha,n}^*=n \left|I_n\right| + \frac{n \left|I_n\right|\Delta_n^2}{\sigma_0^2 \kappa_n^4} -2 \log\left(\alpha \left|I_n\right|\right) + 2 \sqrt{n \left|I_n\right|\left(1+2\frac{\Delta_n^2}{\sigma_0^2\kappa_n^4}\right)\log\left(\frac{1}{\alpha\left|I_n\right|}\right)}.
\end{equation}
Assume furthermore that $\Delta_n,\left|I_n\right|$ and $\kappa_n$ satisfy the following condition:
\begin{equation}\label{eq:cond_LR_test}
\begin{aligned}
n\left|I_n\right| &\left(\kappa_n^4 + 2 \frac{\Delta_n^2}{\sigma_0^2} \right) + \frac{\kappa_n^2 \Delta_n^2 n \left|I_n\right|}{\sigma_0^2}  \\[0.1cm]
\ge &\quad 2 \kappa_n^2 \log\left(\frac1{\left|I_n\right|}\right) + 2 \kappa_n^2 \log\left(\frac1\alpha\right) + 2 \sqrt{n\left|I_n\right| \left(\kappa_n^4 + 2 \frac{\Delta_n^2}{\sigma_0^2} \right) \log\left(\frac{1}{\alpha \left|I_n\right|}\right)} \\[0.1cm]
&\quad+2 \left(1+\kappa_n^2\right) \sqrt{n\left|I_n\right|\left(\kappa_n^4 + 2 \left(1+\kappa_n^2\right) \frac{\Delta_n^2}{\sigma_0^2}\right) \log\left(\frac1\alpha\right)}
\end{aligned}
\end{equation}
Then the test \eqref{eq:LR_test} with the statistic given $T_n$ defined in \eqref{eq:T_statistic} and the threshold \eqref{eq:threshold} satisfies
\[
\limsup\limits_{n \to \infty} \bar\alpha\left(\Phi_n\right) \leq \alpha \qquad\mbox{and}\qquad \limsup\limits_{n \to \infty} \bar\beta\left(\Phi_n\right) \leq \alpha.
\]
\end{theorem}

This theorem allows us to analyze the upper bounds obtained by the likelihood ratio test in the regimes DMR, ER and DVR.

\begin{remark}\label{rem:location_known}
Suppose for a second that not only the width $\left|I_n\right|$, but also the location $I_n$ of the bump is known. In this case, the alternative becomes simple, and the analyzed likelihood ratio test will be optimal as it is the Neyman-Pearson test. It can readily be seen from the proofs that all $\log\left(\left|I_n\right|\right)$-terms vanish in this situation, whereas the other expressions stay the same. Thus our analysis will also determine the detection boundary in this case.
\end{remark}

\section{Lower bounds for the detection boundary}\label{sec:lower_bounds}

We will now determine lower bounds in the three different regimes by analyzing \eqref{eq:cond}.

\subsection{Dominant mean regime (DMR)}

As mentioned before, we expect the same lower bounds as for the homogeneous situation here. 
\begin{theorem}\label{thm:db_signal}
Assume the HBR model with $\left|I_n\right|\searrow 0$ and let $\left(\eps_n\right)$ be any sequence such that $\eps_n \to 0, \eps_n \sqrt{-\log\left(\left|I_n\right|\right)} \to \infty$.
\begin{enumerate}
\item If $\sigma_{n}^2 / \Delta_n \to 0$ and $\sigma_{n}^2=o \left(\eps_n\right)$ as $n \to \infty$, then the sequence $\calS$ with
\begin{equation}\label{eq:lower_bound_signal}
\sqrt{n \left|I_n\right|} \Delta_n  \precsim \left(\sqrt{2}\sigma_0 - \eps_n\right) \sqrt{-\log\left(\left|I_n\right|\right)}
\end{equation} is undetectable.
\item If $\sigma_{n}^2 / \Delta_n \to 0$ where $\sigma_{n}^2 = \sigma^2\left(1+o\left(\eps_n\right)\right)$ and $1/\Delta_n^2 = o \left(\eps_n\right)$ as $n \to \infty$, then the sequence $\calS$ with \eqref{eq:lower_bound_signal} is undetectable.
\end{enumerate}
\end{theorem}

\subsection{Equilibrium regime (ER)}

Now let us consider the case that $ \sigma_{n}^2$ and $\Delta_n$ are asymptotically of the same order. Consequently $\kappa_n^2$ and $\Delta_n$ will be of the same order. In this situation we expect a gain by the additional information coming from $\sigma_n^2>0$. In fact, the following Theorem states that the constant in the detection boundary changes, but the detection rate stays the same:
\begin{subequations}\label{eqs:detection_boundary_equilibrium}
\begin{theorem}\label{thm:db_equilibrium}
Assume the HBR model with $\left|I_n\right|\searrow 0$ and let $\left(\eps_n\right)$ be any sequence such that $\eps_n \to 0, \eps_n \sqrt{-\log\left(\left|I_n\right|\right)} \to \infty$.
\begin{enumerate}
\item Let $\sigma_{n}^2 = c \sigma_0 \Delta_n\left(1 + o\left(\eps_n\right)\right)$, $c>0$ and $\sigma_{n}^2 = o\left(\eps_n\right)$ as $n \to \infty$. Then the sequence $\calS$ with
\begin{equation}\label{eq:lower_bound_equilibrium1}
\sqrt{n \left|I_n\right|} \Delta_n  \precsim \left(C - \eps_n\right) \sqrt{-\log\left(\left|I_n\right|\right)}, \qquad C:= \sqrt{2}\sigma_0 \sqrt{\frac{2}{2+c^2}}
\end{equation}
as $n \to \infty$ is undetectable.
\item If $\sigma_{n}^2 = \sigma^2\left(1+o\left(\eps_n\right)\right)$ and $\Delta_n = \frac{\sigma^2}{c\sigma_0} \left(1+o\left(\eps_n\right)\right)$ as $n \to \infty$, then with $\kappa := \sigma / \sigma_0$ the sequence $\calS$ with
\begin{equation}\label{eq:lower_bound_equilibrium2}
\sqrt{n \left|I_n\right|} \precsim \left(C - \eps_n\right) \sqrt{-\log\left(\left|I_n\right|\right)}, \qquad C:= \frac{1}{\sqrt{\frac{\kappa^2}{2c^2} \left(\kappa^2 + c^2\right) - \frac12\log\left(1+\kappa^2\right)}}
\end{equation}
as $n \to \infty$ is undetectable.
\end{enumerate}
\end{theorem}
\end{subequations}

In the first case, for $c = 0$, we have $C = \sqrt{2} \sigma_0$, which corresponds to no change in the variance and hence reduces to the homogeneous model. But if $c>0$, we always obtain $C <  \sqrt{2} \sigma_0$, more precisely we improve by the multiplicative factor $\sqrt{\frac{2}{2+c^2}} < 1$.

As $C = \left(\frac{\kappa^2}{2c^2} \left(\kappa^2 + c^2\right) - \frac12\log\left(1+\kappa^2\right)\right)^{-1/2}$ is not a multiple of $\sqrt{2} \sigma_0$, the gain in the second case \eqref{eq:lower_bound_equilibrium2} is not that obvious. But using $\log\left(1+\kappa^2\right) \leq \kappa^2$ with equality if and only if $\kappa = 0$ implies that $C \leq \sqrt{2}\sigma_0$ with equality if and only if $\kappa = 0$. 

Finally, if $\sigma_{n} \to 0$, then by a Taylor expansion it follows that
\[
\frac{\kappa^2}{2c^2} \left(\kappa^2 + c^2\right) - \frac12\log\left(1+\kappa^2\right)  = \Delta_n^2 \frac{2+c^2}{4\sigma_0^2} + \mathcal O \left(\sigma_{n}^6\right), \qquad \sigma_{n} \to 0,
\]
and hence \eqref{eq:lower_bound_equilibrium2} will reduce to \eqref{eq:lower_bound_equilibrium1} in this situation.

\subsection{Dominant variance regime (DVR)}

Now let us consider the case that the jump in the mean $\Delta_n$ vanishes faster than the jump in the variance $\sigma_{n}^2$. In this situation we again expect a further gain by the additional information. Somewhat surprisingly, we will even obtain a gain in the detection rate compared to the ER in \eqref{eqs:detection_boundary_equilibrium}. 
\begin{subequations}\label{eqs:detection_boundary_variance}
\begin{theorem}\label{thm:db_variance}
Assume the HBR model with $\left|I_n\right|\searrow 0$ and let $\left(\eps_n\right)$ be any sequence such that $\eps_n \to 0, \eps_n \sqrt{-\log\left(\left|I_n\right|\right)} \to \infty$.
\begin{enumerate}
\item If $\sigma_0\Delta_n = \sigma_{n}^2\theta_n$ with sequences $\Delta_n, \sigma_{n}, \theta_n \to 0$ as $n \to \infty$ where $\sigma_{n}^2 = o\left(\eps_n\right)$ and $\theta_n^2 = o \left(\eps_n\right)$, then the sequence $\calS$ with
\begin{equation}\label{eq:lower_bound_variance1}
\sqrt{n \left|I_n\right|} \Delta_n  \precsim \left(2\sigma_0 - \eps_n\right) \sqrt{-\log\left(\left|I_n\right|\right)}\theta_n
\end{equation}
is undetectable.
\item If $\sigma_{n}= \sigma\left(1+ o\left(\eps_n\right)\right)$ and $\Delta_n^2 = o \left(\eps_n\right)$ as $n \to \infty$, then the sequence $\calS$ with
\begin{equation}\label{eq:lower_bound_variance2}
\sqrt{n \left|I_n\right|} \precsim \left(C - \eps_n\right) \sqrt{-\log\left(\left|I_n\right|\right)}, \qquad C:= \frac{1}{\sqrt{\frac{\kappa^2}{2} - \frac12\log \left(1+\kappa^2\right)}}
\end{equation}
is undetectable.
\end{enumerate}
\end{theorem}
\end{subequations}

First note that \eqref{eq:lower_bound_variance1} can be equivalently formulated as 
\begin{equation}\label{eq:dbv1_ref}
\sqrt{n \left|I_n\right|} \kappa_n^2  \precsim \left(2 - \eps_n\right) \sqrt{-\log\left(\left|I_n\right|\right)},
\end{equation}
which is also meaningful if $\Delta_n \equiv 0$ and hence determines the detection boundary if only a jump in the variance with homogeneous mean $\mu \equiv 0$ occurs. We want to emphasize that the exponent $2$ of $\kappa_n$ in \eqref{eq:dbv1_ref} seems natural, as testing the variance for a change is equivalent to testing the mean of the transformed quantities $\left(Y_i - \Delta_n 1_{I_n}\left(i/n\right)\right)^2$ (see \citep[Sect. 2.8.7.]{ch97}). Consequently, the detection rate improves compared to the DMR \eqref{eq:lower_bound_signal}, as we assume $\kappa_n^2/\Delta_n \to \infty$. 

Note that we can also rewrite \eqref{eq:lower_bound_equilibrium1} in terms of the $\kappa_n^2$-rate, which gives the equivalent expression
\begin{equation}\label{eq:lower_bound_equilibrium1_equiv}
\sqrt{n \left|I_n\right|} \kappa_n^2  \precsim \left(C - \eps_n\right) \sqrt{-\log\left(\left|I_n\right|\right)}, \qquad C := 2 \sqrt{\frac{c^2}{2 + c^2}}.
\end{equation}
This makes the ER comparable with the DVR. As expected $C \leq 2$, and thus we see a clear improvement in the constant over \eqref{eq:dbv1_ref}. If $c \to \infty$, then $C$ tends to $2$ which coincides with \eqref{eq:dbv1_ref}.

Again, if $\kappa \to 0$ it can be seen by a Taylor expansion that \eqref{eq:lower_bound_variance2} reduces to \eqref{eq:lower_bound_variance1}. Furthermore, \eqref{eq:lower_bound_variance2} can be seen as the natural extension of \eqref{eq:lower_bound_equilibrium2} as $c \to \infty$. 

\subsection{Discussion}
The condition $\sigma_{n}^2 = o \left(\eps_n\right)$ seems to be unavoidable as soon as a Taylor expansion (cf. \eqref{eq:taylor2}) is about to be used.

In several assumptions we require a convergence behavior of the form $1 + o \left(\eps_n\right)$. Note that this cannot readily be replaced by $1 + o \left(1\right)$, because the detection boundary is well-defined only up to a term of order $\eps_n$. Thus if the convergence behavior is slower, at least the constants will need to change depending on the specific convergence behavior.

\section{Upper bounds for the detection boundary}\label{sec:upper_bounds}

After determining lower bounds for the detection boundary, we now aim for proving that those are sharp in the following sense: There exists a level $\alpha$ test which detects with asymptotic type II error $\leq \alpha$ bump functions having asymptotic behavior of the form \eqref{eq:lower_bound_signal}, \eqref{eqs:detection_boundary_equilibrium} and \eqref{eqs:detection_boundary_variance} where $\left(C-\eps_n\right)$ on the right-hand side is replaced by $\left(C+\eps_n\right)$.

\subsection{Dominant mean regime (DMR)}

Let us first consider the case $\sigma_{n}^2 / \Delta_n \to 0$ and $\sigma_{n}^2=o \left(\eps_n\right)$ as $n \to \infty$. The following theorem states that our lower bounds are optimal by exploiting the likelihood ratio test:
\begin{theorem}\label{thm:LR_homogeneous}
Assume the HBR model with $\left|I_n\right|\searrow 0$ and let $H_0, H_1^n$ be defined by \eqref{eq:test_prob}. Suppose $\eps_n>0$ is any sequence such that $\eps_n\sqrt{-\log\left(\left|I_n\right|\right)}\to\infty$ as $n\to\infty$ and $\alpha \in \left[0,1\right]$ is a given significance level.

Furthermore, let 
\begin{enumerate}
\item either $\sigma_{n}^2 / \Delta_n \to 0$ and $\sigma_{n}^2=o \left(\eps_n\right)$ 
\item or $\sigma_{n}^2 / \Delta_n \to 0$ and $\sigma_{n}^2= \sigma^2 \left(1+o \left(\eps_n\right)\right)$ 
\end{enumerate}
as $n \to \infty$. If
\begin{equation}\label{eq:condition_homogeneous}
\sqrt{n \left|I_n\right|} \Delta_n \succsim \left(\sqrt{2}\sigma_0 + \eps_n\right) \sqrt{-\log\left(\left|I_n\right|\right)},
\end{equation}
then the test \eqref{eq:LR_test} with $T_n$ defined in \eqref{eq:T_statistic} and the threshold \eqref{eq:threshold} satisfies
\[
\limsup\limits_{n \to \infty} \bar\alpha\left(\Phi_n\right) \leq \alpha \qquad\mbox{and}\qquad \limsup\limits_{n \to \infty} \bar\beta\left(\Phi_n\right) \leq \alpha.
\]
\end{theorem}

\subsection{Equilibrium regime (ER)}

\begin{theorem}\label{thm:LR_equilibrium}
Assume the HBR model with $\left|I_n\right|\searrow 0$ and let $H_0, H_1^n$ be defined by \eqref{eq:test_prob}. Suppose $\eps_n>0$ is any sequence such that $\eps_n\sqrt{-\log\left(\left|I_n\right|\right)}\to\infty$ as $n\to\infty$ and $\alpha \in \left[0,1\right]$ is a given significance level.

Furthermore let $\sigma_{n}^2 =c\sigma_0 \Delta_n \left(1+o\left(\eps_n\right)\right)$ with $c>0$ and $\sigma_{n}^2=o \left(\eps_n\right)$ as $n \to \infty$ and 
\begin{equation}\label{eq:condition_equilibrium}
\sqrt{n \left|I_n\right|} \Delta_n  \succsim \left(C + \eps_n\right) \sqrt{-\log\left(\left|I_n\right|\right)}, \qquad C:= \sqrt{2} \sigma_0 \sqrt{\frac{2}{2+c^2}}.
\end{equation}
Then the test \eqref{eq:LR_test} with $T_n$ defined in \eqref{eq:T_statistic} and the threshold \eqref{eq:threshold} satisfies
\[
\limsup\limits_{n \to \infty} \bar\alpha\left(\Phi_n\right) \leq \alpha \qquad\mbox{and}\qquad \limsup\limits_{n \to \infty} \bar\beta\left(\Phi_n\right) \leq \alpha.
\]
\end{theorem}

Comparing Theorems \ref{thm:db_equilibrium} and \ref{thm:LR_equilibrium} one may note that our upper bounds so far do not handle the case where $\kappa_n^2$ does not tend to $0$. In fact we can also obtain an upper bound from \eqref{eq:cond_LR_test} for $\sigma_{n}^2 = \sigma^2\left(1+o\left(\eps_n\right)\right)$ and $\Delta_n = \frac{\sigma^2}{c\sigma_0} \left(1+o\left(\eps_n\right)\right)$ as $n \to \infty$, i.e.
\begin{equation}\label{eq:er_relaxed_ub}
\sqrt{n \left|I_n\right|} \succsim \left(C + \eps_n\right) \sqrt{-\log\left(\left|I_n\right|\right)}, \qquad C:= \frac{\sqrt{2 \kappa^2 + \frac{4\kappa^2}{c^2} + \frac{2\kappa^4}{c^2} + 1 + \frac{2}{c^2}} + \sqrt{1 + \frac{2}{c^2}}}{\kappa^2 + \frac{2\kappa^2}{c^2} + \frac{\kappa^4}{c^2}}.
\end{equation}
Obviously, this bound does not coincide with the lower bounds from Theorem \ref{thm:db_equilibrium}. This is due to the fact that our tail estimates for the $\chi^2$ distribution does not include any logarithmic terms, but our lower bounds do. Thus it is impossible to obtain the same bounds using Lemma \ref{lem:tails} with our technique. 

\subsection{Dominant variance regime (DVR)}

\begin{theorem}\label{thm:LR_dominantvar}
Assume the HBR model with $\left|I_n\right|\searrow 0$ and let $H_0, H_1^n$ be defined by \eqref{eq:test_prob}. Suppose $\eps_n>0$ is any sequence such that $\eps_n\sqrt{-\log\left(\left|I_n\right|\right)}\to\infty$ as $n\to\infty$ and $\alpha \in \left[0,1\right]$ is a given significance level.

Furthermore let $\sigma_0\Delta_n = \sigma_{n}^2\theta_n$ with sequences $\Delta_n, \sigma_{n}, \theta_n \to 0$ as $n \to \infty$ where $\sigma_{n}^2=o \left(\eps_n\right)$, $\theta_n = o \left(\eps_n\right)$ and
\begin{equation}\label{eq:condition_dominantvar}
\sqrt{n \left|I_n\right|} \Delta_n  \succsim \left(2\sigma_0 + \eps_n\right) \sqrt{-\log\left(\left|I_n\right|\right)}\theta_n
\end{equation}
Then the test \eqref{eq:LR_test} with $T_n$ defined in \eqref{eq:T_statistic} and the threshold \eqref{eq:threshold} satisfies
\[
\limsup\limits_{n \to \infty} \bar\alpha\left(\Phi_n\right) \leq \alpha \qquad\mbox{and}\qquad \limsup\limits_{n \to \infty} \bar\beta\left(\Phi_n\right) \leq \alpha.
\]
\end{theorem}

Note again that we can also obtain upper bounds in case that $\sigma_{n}= \sigma\left(1+ o\left(\eps_n\right)\right)$ and $\Delta_n^2 = o \left(\eps_n\right)$ as $n \to \infty$, i.e.
\begin{equation}\label{eq:dvr_relaxed_ub}
\sqrt{n \left|I_n\right|} \succsim \left(C + \eps_n\right) \sqrt{-\log\left(\left|I_n\right|\right)}, \qquad C:= \frac{\sqrt{2\kappa^2 + 1}+1}{\kappa^2},
\end{equation}
which again cannot coincide with the lower bounds from Theorem \ref{thm:db_variance} for the same reason as in the equilibrium regime.

\subsection{Adaptation}\label{sec:adaptivity}

In this Section we will discuss an adaptive version of the likelihood ratio test for which $\Delta_n$ does not need to be known. In this case it is natural to replace $\Delta_n$ by its empirical version $\hat \Delta_n := \left(n \left|I_n\right|\right)^{-1} \sum_{i: \frac{i}{n} \in I_{n}} Y_i$. This leads to the marginal likelihood ratio
\begin{align*}
L_{n}\left(I_n,\kappa_n;Y\right) &=\left(\kappa_n^2+1\right)^{-\frac{n\left|I_n\right|}{2}}  \exp\left(\frac{\kappa_n^2}{2\sigma_0^2\left(\kappa_n^2+1\right)}\sum\limits_{i: \frac{i}{n} \in I_n} Y_i^2+\frac{n\left|I_n\right|}{2\sigma_0^2\left(\kappa_n^2+1\right)} \hat \Delta_n^2\right) \\[0.1cm]
&= \left(\kappa_n^2+1\right)^{-\frac{n\left|I_n\right|}{2}}  \exp\left(\frac{\kappa_n^2}{2\sigma_0^2\left(\kappa_n^2+1\right)}\sum\limits_{i: \frac{i}{n} \in I_n} \left(Y_i - \hat \Delta_n\right)^2 + \frac{n\left|I_n\right|}{2 \sigma_0^2} \hat \Delta_n^2\right)
\end{align*}
and to its corresponding test statistic
\begin{equation}\label{eq:T_statistic_adaptive}
T_n^{\kappa_n,\left|I_n\right|,\sigma_0} \left(Y\right) := \sup\limits_{A_n \in \mathcal A_n} 
\left(\frac{\kappa_n^2}{\sigma_0^2\left(\kappa_n^2+1\right)}\sum\limits_{i: \frac{i}{n} \in A_n} \left(Y_i - \hat \Delta_n\right)^2 + \frac{1}{\sigma_0^2n\left|I_n\right|} \left(~\sum_{i: \frac{i}{n} \in A_{n}} Y_i \right)^2\right).
\end{equation}
Again we will drop the dependence on the parameters subsequently and write $T_n^* := T_n^{\kappa_n,\left|I_n\right|,\sigma_0}$ for the statistic in \eqref{eq:T_statistic_adaptive} to ease the presentation. Let
\begin{equation}\label{eq:LR_adaptive_test}
\Phi_n^*\left(Y\right) := \begin{cases} 1 & \text{if } T_n^* \left(Y\right) > c_{\alpha,n}^*, \\[0.1cm] 0 & \text{else} \end{cases}
\end{equation}
be the corresponding adaptive test where $c_{\alpha,n}^* \in \mathbb R$ is determined by the test level $\alpha$. An analysis of the likelihood ratio test \eqref{eq:LR_adaptive_test} with $T_n^*$ as in \eqref{eq:T_statistic_adaptive} can be carried out similarly to the non-adaptive case. The main difference is that we now need Lemma \ref{lem:tails} with $k = 2$. Thus, the following result can be proven:
\begin{theorem}\label{thm:LR_adaptive}
Assume the HBR model with $\left|I_n\right|\searrow 0$ and let $H_0, H_1^n$ be defined by \eqref{eq:test_prob}. Suppose $\eps_n>0$ is any sequence such that $\eps_n\sqrt{-\log\left(\left|I_n\right|\right)}\to\infty$ as $n\to\infty$ and $\alpha \in \left[0,1\right]$ is a given significance level. Define the corresponding threshold by
\begin{equation}\label{eq:threshold_adaptive}
c_{n,\alpha}^* := \frac{\kappa_n^2}{\kappa_n^2 +1} (n\left|I_n\right|-1)+1 +2\sqrt{\left[\frac{\kappa_n^4}{\left(\kappa_n^2 +1\right)^2}\left(n\left|I_n\right|-1\right) + 1\right]\log\left(\frac{1}{\alpha\left|I_n\right|}\right)} + 2\log\left(\frac{1}{\alpha\left|I_n\right|}\right)
\end{equation}
Now suppose that we are in one of the following three situations:
\begin{itemize}
\item[DMR:]
$\kappa_n^2 / \Delta_n \to 0$ and $1 / \Delta_n  = o \left(\eps_n\right)$ as $n \to \infty$ and 
\begin{equation}\label{eq:adaptive_condition_homogeneous}
\sqrt{n \left|I_n\right|} \Delta_n  \succsim \left(\sqrt{2}\sigma_0 + \eps_n\right) \sqrt{-\log\left(\left|I_n\right|\right)}.
\end{equation}
\item[ER:]
$\sigma_{n}^2 =c\sigma_0 \Delta_n \left(1+o\left(\eps_n\right)\right)$ with $c>0$ and $\sigma_{n}^2=o \left(\eps_n\right)$ as $n \to \infty$ and 
\begin{equation}\label{eq:adaptive_condition_equilibrium}
\sqrt{n \left|I_n\right|} \Delta_n  \succsim \left(C + \eps_n\right) \sqrt{-\log\left(\left|I_n\right|\right)}, \qquad C:= \sigma_0 \frac{c+ \sqrt{2+3c^2}}{1+c^2}.
\end{equation}
\item[DVR:]
$\sigma_0\Delta_n = \sigma_{n}^2\theta_n$ with sequences $\Delta_n, \sigma_{n}, \theta_n \to 0$ as $n \to \infty$ where $\sigma_{n}^2=o \left(\eps_n\right)$, $\theta_n = o \left(\eps_n\right)$ and furthermore
\begin{equation}\label{eq:adaptive_condition_dominantvar}
\sqrt{n \left|I_n\right|} \Delta_n  \succsim \left(\left(1+\sqrt{3}\right)\sigma_0 + \eps_n\right) \sqrt{-\log\left(\left|I_n\right|\right)}\theta_n.
\end{equation}
\end{itemize}
In any of these cases the test \eqref{eq:LR_adaptive_test} with the statistic given $T_n^*$ defined in \eqref{eq:T_statistic_adaptive} and the threshold \eqref{eq:threshold_adaptive} satisfies
\[
\limsup\limits_{n \to \infty} \bar\alpha\left(\Phi_n^*\right) \leq \alpha \qquad\mbox{and}\qquad \limsup\limits_{n \to \infty} \bar\beta\left(\Phi_n^*\right) \leq \alpha.
\]
\end{theorem}

The above theorem states upper bounds which are adaptive to the height $\Delta_n$ of the bump. It is remarkable that the adaptive upper bound coincides with the non-adaptive one (cf. Theorem \ref{thm:LR_homogeneous}) in the dominant mean regime, where in the two other regimes the rates coincide but we lose a constant. 

\section{Simulations}\label{sec:simulations}

In this Section we will examine the finite sample properties of the test discussed in Section \ref{sec:upper_bounds} by means of a simulation study. We therefore implemented the non-adaptive likelihood ratio test $\Phi_n$ (cf. \eqref{eq:LR_test}) with the threshold $c_{\alpha,n}^*$ as in \eqref{eq:threshold} as well as the adaptive likelihood ratio test $\Phi_n^*$ (cf. \eqref{eq:LR_adaptive_test}) with $c_{\alpha,n}^*$ as in \eqref{eq:threshold_adaptive}. For these tests the levels $\bar\alpha\left(\Phi_n\right), \bar\alpha\left(\Phi_n^*\right)$ and powers $1-\bar\beta\left(\Phi_n\right), 1-\bar\beta\left(\Phi_n^*\right)$ have been examined by $10^4$ simulation runs for $\alpha = 0.05$ with different parameters $\sigma_0$, $n$, $\kappa_n$, $\Delta_n$ and $\left|I_n\right|$, respectively.

All derived detection boundaries involve the ratio $n \left|I_n\right| / \log \left(1/\left|I_n\right|\right)$. For the simulations below, we therefore fixed this ratio and then considered three situations, namely small sample size ($\left|I_n\right| = 0.1$, $n = 829$), medium sample size ($\left|I_n\right| = 0.05$, $n = 2157$) and large sample size ($\left|I_n\right| = 0.025$, $n = 5312$). In all three situations we computed the power for $\sigma_0 = 1$, $\kappa_n^2 \in \left\{0.01,0.02,...,1.2\right\}$ and $\Delta_n \in \left\{0.01,0.02,...,0.7\right\}$. The corresponding results are shown in Figure \ref{fig:simulations}. 

\begin{figure}[!htb]
\setlength\fheight{3.5cm} \setlength\fwidth{3.5cm} 
\centering
\subfigure[non-adaptive test power, small sample case]{
\label{subfig:small_nonadaptive}
% This file was created by matlab2tikz v0.4.7 running on MATLAB 7.11.
% Copyright (c) 2008--2014, Nico Schlömer <nico.schloemer@gmail.com>
% All rights reserved.
% Minimal pgfplots version: 1.3
% 
% The latest updates can be retrieved from
%   http://www.mathworks.com/matlabcentral/fileexchange/22022-matlab2tikz
% where you can also make suggestions and rate matlab2tikz.
% 
\begin{tikzpicture}

\begin{axis}[%
width=\fwidth,
height=\fheight,
axis on top,
scale only axis,
xmin=0.005,
xmax=0.705,
xlabel={$\Delta_n$},
ymin=0.005,
ymax=1.205,
ylabel={$\kappa_n^2$},
colormap={mymap}{[1pt] rgb(0pt)=(1,1,1); rgb(16pt)=(1,1,0); rgb(40pt)=(1,0,0); rgb(63pt)=(0.0416667,0,0)},
colorbar,
point meta min=0,
point meta max=1
]
\addplot [forget plot] graphics [xmin=0.005,xmax=0.705,ymin=0.005,ymax=1.205] {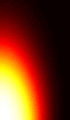};
\end{axis}
\end{tikzpicture}% 
}
\subfigure[adaptive test power, small sample case]{
\label{subfig:small_adaptive}
% This file was created by matlab2tikz v0.4.7 running on MATLAB 7.11.
% Copyright (c) 2008--2014, Nico Schlömer <nico.schloemer@gmail.com>
% All rights reserved.
% Minimal pgfplots version: 1.3
% 
% The latest updates can be retrieved from
%   http://www.mathworks.com/matlabcentral/fileexchange/22022-matlab2tikz
% where you can also make suggestions and rate matlab2tikz.
% 
\begin{tikzpicture}

\begin{axis}[%
width=\fwidth,
height=\fheight,
axis on top,
scale only axis,
xmin=0.005,
xmax=0.705,
xlabel={$\Delta_n$},
ymin=0.005,
ymax=1.205,
ylabel={$\kappa_n^2$},
colormap={mymap}{[1pt] rgb(0pt)=(1,1,1); rgb(16pt)=(1,1,0); rgb(40pt)=(1,0,0); rgb(63pt)=(0.0416667,0,0)},
colorbar,
point meta min=0,
point meta max=1
]
\addplot [forget plot] graphics [xmin=0.005,xmax=0.705,ymin=0.005,ymax=1.205] {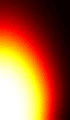};
\end{axis}
\end{tikzpicture}% 
}\\
\subfigure[non-adaptive test power, medium sample case]{
\label{subfig:medium_nonadaptive}
% This file was created by matlab2tikz v0.4.7 running on MATLAB 7.11.
% Copyright (c) 2008--2014, Nico Schlömer <nico.schloemer@gmail.com>
% All rights reserved.
% Minimal pgfplots version: 1.3
% 
% The latest updates can be retrieved from
%   http://www.mathworks.com/matlabcentral/fileexchange/22022-matlab2tikz
% where you can also make suggestions and rate matlab2tikz.
% 
\begin{tikzpicture}

\begin{axis}[%
width=\fwidth,
height=\fheight,
axis on top,
scale only axis,
xmin=0.005,
xmax=0.705,
xlabel={$\Delta_n$},
ymin=0.005,
ymax=1.205,
ylabel={$\kappa_n^2$},
colormap={mymap}{[1pt] rgb(0pt)=(1,1,1); rgb(16pt)=(1,1,0); rgb(40pt)=(1,0,0); rgb(63pt)=(0.0416667,0,0)},
colorbar,
point meta min=0,
point meta max=1
]
\addplot [forget plot] graphics [xmin=0.005,xmax=0.705,ymin=0.005,ymax=1.205] {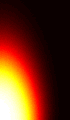};
\end{axis}
\end{tikzpicture}% 
}
\subfigure[adaptive test power, medium sample case]{
\label{subfig:medium_adaptive}
% This file was created by matlab2tikz v0.4.7 running on MATLAB 7.11.
% Copyright (c) 2008--2014, Nico Schlömer <nico.schloemer@gmail.com>
% All rights reserved.
% Minimal pgfplots version: 1.3
% 
% The latest updates can be retrieved from
%   http://www.mathworks.com/matlabcentral/fileexchange/22022-matlab2tikz
% where you can also make suggestions and rate matlab2tikz.
% 
\begin{tikzpicture}

\begin{axis}[%
width=\fwidth,
height=\fheight,
axis on top,
scale only axis,
xmin=0.005,
xmax=0.705,
xlabel={$\Delta_n$},
ymin=0.005,
ymax=1.205,
ylabel={$\kappa_n^2$},
colormap={mymap}{[1pt] rgb(0pt)=(1,1,1); rgb(16pt)=(1,1,0); rgb(40pt)=(1,0,0); rgb(63pt)=(0.0416667,0,0)},
colorbar,
point meta min=0,
point meta max=1
]
\addplot [forget plot] graphics [xmin=0.005,xmax=0.705,ymin=0.005,ymax=1.205] {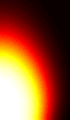};
\end{axis}
\end{tikzpicture}% 
}\\
\subfigure[non-adaptive test power, large sample case]{
\label{subfig:large_nonadaptive}
% This file was created by matlab2tikz v0.4.7 running on MATLAB 7.11.
% Copyright (c) 2008--2014, Nico Schlömer <nico.schloemer@gmail.com>
% All rights reserved.
% Minimal pgfplots version: 1.3
% 
% The latest updates can be retrieved from
%   http://www.mathworks.com/matlabcentral/fileexchange/22022-matlab2tikz
% where you can also make suggestions and rate matlab2tikz.
% 
\begin{tikzpicture}

\begin{axis}[%
width=\fwidth,
height=\fheight,
axis on top,
scale only axis,
xmin=0.005,
xmax=0.705,
xlabel={$\Delta_n$},
ymin=0.005,
ymax=1.205,
ylabel={$\kappa_n^2$},
colormap={mymap}{[1pt] rgb(0pt)=(1,1,1); rgb(16pt)=(1,1,0); rgb(40pt)=(1,0,0); rgb(63pt)=(0.0416667,0,0)},
colorbar,
point meta min=0,
point meta max=1
]
\addplot [forget plot] graphics [xmin=0.005,xmax=0.705,ymin=0.005,ymax=1.205] {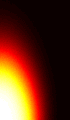};
\end{axis}
\end{tikzpicture}% 
}
\subfigure[adaptive test power, large sample case]{
\label{subfig:large_adaptive}
% This file was created by matlab2tikz v0.4.7 running on MATLAB 7.11.
% Copyright (c) 2008--2014, Nico Schlömer <nico.schloemer@gmail.com>
% All rights reserved.
% Minimal pgfplots version: 1.3
% 
% The latest updates can be retrieved from
%   http://www.mathworks.com/matlabcentral/fileexchange/22022-matlab2tikz
% where you can also make suggestions and rate matlab2tikz.
% 
\begin{tikzpicture}

\begin{axis}[%
width=\fwidth,
height=\fheight,
axis on top,
scale only axis,
xmin=0.005,
xmax=0.705,
xlabel={$\Delta_n$},
ymin=0.005,
ymax=1.205,
ylabel={$\kappa_n^2$},
colormap={mymap}{[1pt] rgb(0pt)=(1,1,1); rgb(16pt)=(1,1,0); rgb(40pt)=(1,0,0); rgb(63pt)=(0.0416667,0,0)},
colorbar,
point meta min=0,
point meta max=1
]
\addplot [forget plot] graphics [xmin=0.005,xmax=0.705,ymin=0.005,ymax=1.205] {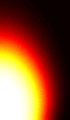};
\end{axis}
\end{tikzpicture}% 
}
\label{fig:simulations}
\caption{Power of non-adaptive (known mean $\Delta_n$) and adaptive ($\Delta_n$ unknown) tests plotted for $\Delta_n$ vs. $\kappa_n^2$ in three different situations (small sample size $\left|I_n\right| = 0.1$, $n = 829$, medium sample size $\left|I_n\right| = 0.05$, $n = 2157$, and large sample size $\left|I_n\right| = 0.025$, $n = 5312$), each simulated by $10^4$ Monte Carlo experiments.}
\end{figure}

\paragraph{Power loss caused by adaptivity.} By comparing the constants in Theorems \ref{thm:db_signal}--\ref{thm:db_variance} and Theorem \ref{thm:LR_adaptive} it can be seen (as discussed in Section \ref{sec:overview}) that the adaptive test is (asymptotically) at least as powerful as non-adaptive test if $\Delta_n$ and $\kappa_n^2$ are multiplied by a factor of $\sqrt{2}$. By comparing subplots \ref{subfig:small_adaptive} and \ref{subfig:small_nonadaptive}, \ref{subfig:medium_adaptive} and \ref{subfig:medium_nonadaptive} and \ref{subfig:large_adaptive} and \ref{subfig:large_nonadaptive} respectively, it can be seen roughly that this also holds true in the simulations. Interestingly, the shape of the function $r$ from Figure \ref{fig:adaptation} becomes also visible as a small 'belly' on the lower right in the adaptive large sample situation (subplot \ref{subfig:large_adaptive}), i.e. $\Delta_n$ has initially to increase with $\kappa_n$ to reach the same power. 

\paragraph{Improvement by heterogeneity.} Let us consider an exemplary situation from the medium sample size case, i.e. $n = 2157$ and $\left|I_n\right| = 0.05$. From subplot \ref{subfig:medium_nonadaptive} it can be seen that the empirical power for $\Delta_n = 0.2$ increases dramatically when we alter $\kappa_n$ from $\kappa_n^2 = 0$ to $\kappa_n^2 = 0.5$. For $\kappa_n = 0$ we use the standard likelihood ratio test for comparison, which is based on the statistic
\[
T_n^{\left|I_n\right|}\left(Y\right) = \sup\limits_{I_n \in \mathcal A_n} \frac{1}{\sigma_0^2} \sum\limits_{i: \frac{i}{n} \in I_n} Y_i
\]
and uses the exact quantiles of $\mathcal N \left(0, n \left|I_n\right|\right)$ and the union bound to compute the threshold $c_{\alpha,n}^*$. A comparison of the two situations is shown in Figure \ref{fig:sim_exemplary}. Remarkably, the non-adaptive test at level $5\%$ is able to detect the bump when $\kappa_n^2 = 0.5$, but not when $\kappa_n^2 = 0$ for the depicted data. Note that the empirical power of the test in the heterogeneous situation is $47.1\%$, whereas the one of the test in the homogeneous situation is only $22.5\%$ (note that this is larger than the power depicted in Figure \ref{fig:simulations} for $\kappa_n^2 = 0.01$ due to the used exact quantiles for the homogeneous test).

\begin{figure}[!htb]
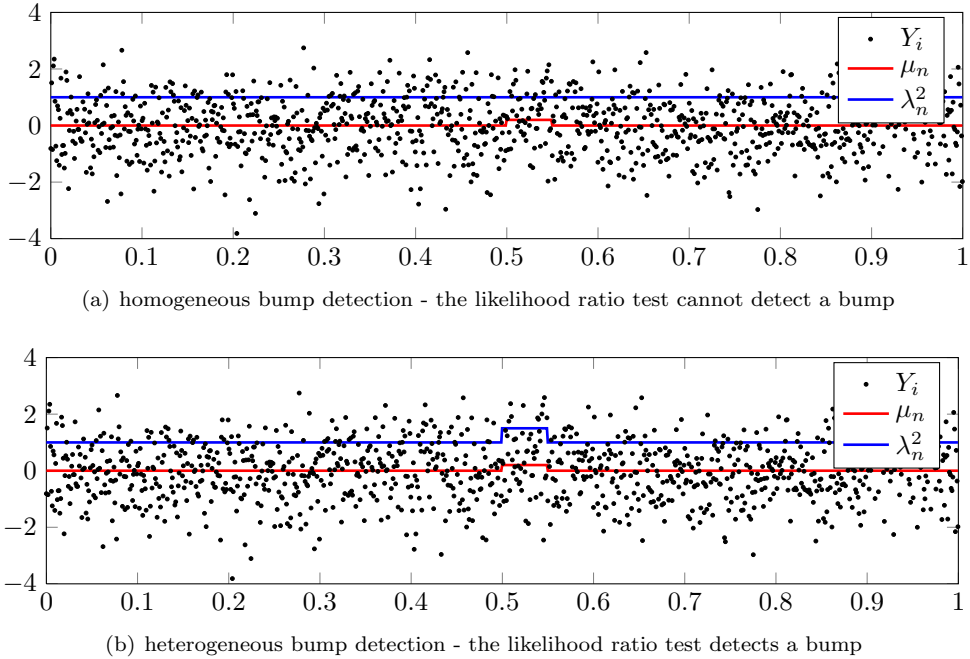

\setlength\fheight{3cm} \setlength\fwidth{12cm}
\centering
\subfigure[homogeneous bump detection - the likelihood ratio test cannot detect a bump]{
\label{subfig:exemplary_homogeneous}
\input{sim_homogeneous.tikz} 
}\\
\subfigure[heterogeneous bump detection - the likelihood ratio test detects a bump]{
\label{subfig:exemplary_heterogeneous}
\input{sim_heterogeneous.tikz} 
}
\label{fig:sim_exemplary}
\caption{Comparison of homogeneous and heterogeneous bump detection. We simulated data from the model \eqref{eq:model} once with $\lambda_n \equiv 1$ (homogeneous) and once with $\lambda_n (x) =  \sqrt{1 + 0.5 \cdot  1_{I_n}(x)}$ (heterogeneous) for the same errors $Z_i$, $1 \leq i \leq n$. Here $n = 2157$, $I_n = \left[0.5, 0.55\right]$ and $\Delta_n = 0.2$. Even though the observations look quite similar in both situations, the non-adaptive test \eqref{eq:LR_test} with level $5\%$ is able to detect a bump in the heterogeneous situation, but not in the homogeneous situation.}
\end{figure}

\section{Conclusions and future work}\label{sec:conclusions}

\paragraph{Summary.} We obtain asymptotic minimax testing rates for the problem of detecting a bump with unknown location $I_n$ in a heterogeneous Gaussian regression model, the HBR model. The lower bound condition \eqref{eq:cond} guides us for a distinction into three possible regimes in the HBR model. It allows to quantify the influence of the presence of an additional information about the change in the variance on the detection rate. Both cases of known and unknown bump height $\Delta_n$ are considered. In three possible regimes of the HBR model the exact constants of the asymptotic testing risk are obtained for the case of known $\Delta_n$. In the case of adaptation to an unknown $\Delta_n$ the obtained rates remain the same as for the case of known $\Delta_n$. The optimal constant coincides in the DMR with the case for known $\Delta_n$, else it is unknown (see Table~\ref{tab:results} for details). The obtained results are based on non-asymptotic deviation inequalities for the type I and type II errors of testing. We would like to stress that the thresholds \eqref{eq:threshold} and \eqref{eq:threshold_adaptive} of both tests are non-asymptotic in the sense that they guarantee the non-asymptotic significance level~ $\alpha$. This allows to apply the proposed tests even for finite samples. We also provide non-asymptotic upper detection bound conditions \eqref{eq:cond_LR_test} and \eqref{ineq_cT}. Lemma \ref{lem:estimate_1} on the lower bound estimate can be easily proven for the case of a non-asymptotic type I error as well. 

\paragraph{Open issues and extensions.} Our results are incomplete concerning the situation if $\kappa_n \to \kappa >0$ in the equilibrium and the dominant variance regime. We applied the same techniques to prove lower bounds in this case, and we conjecture that these lower bounds are optimal, i.e. determine the detection boundary. We were not able to prove coinciding upper bounds due to limitations in the used deviation inequalities. To do so we would need deviation inequalities for $\chi^2$ random variables involving the logarithmic terms from the moments, but such a result is beyond the scope of this paper.

Furthermore, we only investigated adaptation to $\Delta_n$, but it would also be interesting to see what happens if the ``baseline'' variance $\sigma_0$ or / and $\sigma_n$ is unknown. By using the empirical variance on $I_n$ as an estimator, this would lead to a likelihood ratio test with a test statistic involving fourth powers of Gaussian random variables. For the analysis good tail inequalities for such random variables are on demand.

Additionally, one could also ask for adaptation w.r.t. $\left|I_n\right|$ and corresponding lower bounds, but this is far beyond the scope of this paper and requires further research. See, however, \citep{fms14} for the homogeneous case.

Another important open question is concerned with the situation of multiple jumps. In the homogeneous case this has been addressed in \citep{fms14}, and upper bounds with constants $4$ for a bounded and $12$ for an unbounded number of change points have been proven. Even though it seems likely that these constants are optimal, no corresponding lower bounds have been proven so far. In the present HBR model this is an open question as well.

As stressed in the Introduction we currently only deal with the HBR model, i.e. that a jump in mean also enforces a jump in variance ($\kappa_n>0$). But often it is more reasonable to consider the situation that whenever a jump in mean happens, there \textbf{can} also be a jump in variance, but not conversely. This would be modeled by letting $\kappa_n \in \left[0,\infty\right)$. For simplicity one could restrict oneself to $\kappa_n \in \left\{0,\kappa_{1,n}\right\}$. Still it is unclear what the detection boundary should be in this situation. In fact it is not even clear if the additional uncertainty leads to a loss of information. 

We believe that lower bounds can be constructed in the same way (and it is likely that they stay the same), but the calculation of upper bounds seems quite more involved.

\section*{Acknowledgements}

AM wants to thank A. Goldenshluger for helpful discussions. The authors are grateful for several helpful comments of two referees and the associate editor which lead to an improvement in clarity and diction. AM and FW gratefully acknowledge financial support by the German Research Foundation DFG through subproject A07 of CRC 755, and AM also acknowledges support through FOR 916, subprojects B1, B3, and CRC 803 subproject Z02. We also thank A. Bartsch, O. Sch\"utte and C. Steinem (Institute of Organic and Biomolecular Chemistry, University of G\"ottingen) for providing the ion channel data.

\appendix

\section{A chi-squared deviation inequality}\label{app:tails}

For the analysis of the likelihood ratio test and the adaptive version for unknown $\Delta_n$ in Subsection \ref{sec:adaptivity} a specific deviation inequality for the weighted sum of two non-central $\chi^2$ distributed random variables is required. Recall that $X\sim \chi_d^2\left(a^2\right)$ with non-centrality parameter $a^2$ and $d$ degrees of freedom, if 
\[
X=\sum\limits_{j=1} ^d \left(\xi_j+a_j\right)^2  \qquad\text{where}\qquad \xi_j\stackrel{\text{i.i.d.}}{\sim}\mathcal N\left(0,1\right) \text{ and }a^2=\sum\limits_{j=1}^d a_j^2.
\]
In this case $\E{}{X}=d+a^2$ and $\Var X=2\left(d+2a^2\right)$.

In the following we consider the weighted sum of $k$ non-central chi-squared variables $Z = \sum_{i=1}^k b_i X_i$, where $b_i\geq 0$ and $X_i \sim \chi_{d_i}^2\left(a_i^2\right)$ are independent with $d_i \in \mathbb N$, $a_i^2 \geq 0$, $i=1,\dots,k$. Note that
\[
\E{}{Z} = \sum_{i=1}^k b_i (d_i+a_i^2) \qquad \text{and}\qquad \Var{Z} = 2 \sum_{i=1}^ k b_i^2 (d_i+ 2 a_i^2).
\]

\begin{lemma}\label{lem:tails}
Let $Z = \sum_{i= 1}^k b_i X_i$, where $b_i\geq 0$ and $X_i \sim \chi_{d_i}^2\left(a_i^2\right)$ are independent with $d_i \in \mathbb N$, $a_i^2 \geq 0$, $i=1,\dots,k$. Let $\|b\|_\infty =\max_{1 \leq i\leq k} |b_i|$. Then the following deviation bounds hold true for all $x>0$: 
\begin{subequations}\label{eqs:tails}
\begin{align}
\Prob{}{Z\le \E{}{Z} -\sqrt{2\Var{Z}x}}&\leq \exp\left(-x\right), \label{eq:tail_lower} \\[0.1cm]
\Prob{}{Z> \E{}{Z} +\sqrt{2\Var{Z}x}+2\|b\|_\infty x}&\leq \exp\left(-x\right). \label{eq:tail_upper}
\end{align}
\end{subequations}
\end{lemma}
The above Lemma is a reformulation of \citep[Lemma 2]{llm12}.

\bigskip

In the following we will briefly discuss other deviation inequalities for $\chi^2$ distributed random variables \citep{Birge2001,sz13,rd13,Hsu2012,BenTal2009,blm13} and their connections to \eqref{eqs:tails} from the perspective of our purpose. At first we mention an inequality by Birg\'e \citep[Lemma 8.1]{Birge2001}, which deals with $k =1$ and coincides with the above result in that case. Also related to \eqref{eqs:tails} is Prop. 6 in Rohde \& D\"umbgen \citep{rd13}, 
\begin{equation}\label{eq:dr_tail}
\Prob{}{Z> \E{}{Z}+ \sqrt{2 \Var Z x+4\|b\|_\infty^2 x^2} +2 \|b\|_\infty x} \leq \exp\left(-x\right),\qquad x>0.  
\end{equation}
This differs from our bound \eqref{eq:tail_upper} by the additional $4\|b\|_\infty^2 x^2$ term in the square root, revealing \eqref{eq:tail_upper} as strictly sharper. 

\smallskip

Other deviation results were proven for second order forms $f_w\left(s\right) = s^TB^TBs + w^T s, s \in \mathbb R^d$ of Gaussian vectors with $B \in \mathbb R^{d \times d}, w \in \mathbb R^d$. To relate this with our situation, let $d = \sum_{i=1}^k d_i$, choose any $a_{j,i} \in \mathbb R$ such that $\sum_{j=1}^{d_i} a_{j,i}^2 = a_i^2$ for all $1 \leq i \leq k$ and set
\begin{align*}
B &:= \text{diag}\left(\sqrt{b_1}\cdot \1_{d_1},\sqrt{b_2}\cdot \1_{d_2},...,\sqrt{b_k}\cdot \1_{d_k}\right), \\
w &:= \left(b_1a_{1,1},b_1a_{2,1},...,b_1a_{d_1,1},b_2a_{1,2},...,b_2a_{d_2,2},b_3a_{1,3},...\right)^T.
\end{align*}
Here $\1_{m} = \left(1,...,1\right) \in \mathbb R^m$. Let $\Tr\left(A\right)$ be the trace of a matrix $A$. Then $\Tr\left(B^TB\right) = \Tr\left(B^2\right) = \sum_{i=1}^k b_i d_i$, $\Tr\left(B^4\right) = \sum_{i=1}^k b_i^2 d_i$, $\|B^2\|_\infty = \left\Vert b \right\Vert_{\infty}$ and $w^Tw = \sum_{i=1}^kb_i^2a_i^2$. This yields
\begin{equation}\label{eq:rel_Z}
f_w\left(\zeta\right) = Z - \sum\limits_{i=1}^k b_ia_i^2 = Z - \E{}{Z} + \Tr\left(B^2\right), \qquad \zeta\sim \mathcal N\left(0,I_d\right)
\end{equation}
with $Z$ as in Lemma \ref{lem:tails}. 

A deviation inequality for the case $w \neq 0$ can be found in the monograph by Ben-Tal et al. \citep[Prop.~4.5.10]{BenTal2009}, which provides in our setting the bound
\begin{equation}\label{eq:ben_tal_tail}
\Prob{}{f_w\left(\zeta\right)>\Tr(B^2)+2\sqrt{\|B^2\|^2x^2+2(\Tr(B^4)+w^Tw)x}+2\|B^2\|x}\leq\exp\left(-x\right)
\end{equation}
for all $x>0$, where $\|B^2\|$ denotes the spectral norm of $B^2$ given by the square root of the largest eigenvalue. Taking \eqref{eq:rel_Z} into account, this yields
\[
\Prob{}{Z > \E{}{Z} + 2\sqrt{2\sum\limits_{i=1}^k b_i^2 \left(d_i + a_i^2\right) x + \left\Vert b \right\Vert_\infty^2x^2}+2\left\Vert b \right\Vert_\infty x}\leq\exp\left(-t\right), \qquad t >0
\]
for $Z$ as in Lemma \ref{lem:tails}. As $8\sum_{i=1}^k b_i^2 \left(d_i + a_i^2\right) > 2 \Var{Z} $, this inequality is strictly weaker than the D\"umbgen-Rohde estimate \eqref{eq:dr_tail} in the present setting.

For the case $w = 0$ (corresponding to centered $\chi^2$ random variables or $a_i = 0$ for $1 \leq i \leq k$ in our notation), Lemma \ref{lem:tails} reduces to the deviation inequality in Hsu et al. \citep{Hsu2012}:
\begin{equation}\label{eq:hsu_tail}
\Prob{}{f_0\left(\zeta\right)> \Tr (B^2) + 2 \sqrt{\Tr (B^4) x} + 2\|B^2\|_\infty x} \leq \exp\left(-x\right), \qquad x >0.
\end{equation}
The situation $w = 0$ has also been considered by Spokoiny \& Zhilova \citep{sz13}, where the following inequality is proven:
\begin{equation}\label{eq:sz_tail}
\Prob{}{f_0\left(\zeta\right)> \Tr (B^2) + \max\left\{2 \sqrt{2\Tr (B^4)x}, 6\|B^2\|_\infty x\right\}} \leq \exp\left(-x\right), \qquad x >0
\end{equation}
Comparing \eqref{eq:hsu_tail} with \eqref{eq:sz_tail} gives that \eqref{eq:sz_tail} is sharper for moderate deviations, i.e. if
\[
(3-2\sqrt 2) \frac{\Tr\left(B^4\right)}{\|B^2\|^2_\infty} \leq x \leq \frac14 \frac{\Tr\left(B^4\right)}{\|B^2\|^2_\infty},
\]
for large and small deviations it is weaker. We are actually interested in large deviations, and furthermore the restriction $a_i = 0$ makes \eqref{eq:sz_tail} not applicable for our purpose.

\section{Proofs of Section \ref{sec:method}}

\begin{proof}[Proof of Lemma \ref{lem:estimate_1}]
We follow \citep{ds01},
\begin{align*}
1-\bar\beta \left(\Phi_n\right)- \alpha \leq&\inf\limits_{\left(\mu_n,\lambda_n\right) \in S_n} \E{\mu_{n},\lambda_{n}}{\Phi_n\left(Y\right)} - \alpha \\[0.1cm]
\leq& \inf\limits_{I_n \in \mathcal A_n} \E{\mu_n = \Delta_n 1_{I_n},\lambda_n^2 = \sigma_0^2 + \sigma_n^2 1_{I_n}}{\Phi_n\left(Y\right)} - \alpha \\[0.1cm]
\leq& \frac{1}{l_n} \sum\limits_{I_n \in \mathcal A_n} \left[\E{\mu_n = \Delta_n 1_{I_n},\lambda_n^2 = \sigma_0^2 + \sigma_n^2 1_{I_n}}{\Phi_n\left(Y\right)} -\alpha\right] \\[0.1cm]
\leq&  \frac{1}{l_n} \sum\limits_{I_n \in \mathcal A_n} \E{\mu_n = \Delta_n 1_{I_n},\lambda_n^2 = \sigma_0^2 + \sigma_n^2 1_{I_n}}{\Phi_n\left(Y\right)-\E{\mu\equiv 0}{\Phi_n\left(Y\right)}} + o \left(1\right) \\[0.1cm]
=& \E{\mu \equiv 0,\lambda\equiv\sigma_0}{\left(~\frac{1}{l_n}\sum\limits_{I_n \in \mathcal A_n} L_{n}\left(\Delta_n,I_n,\kappa_n,Y\right) -1\right) \Phi_n\left(Y\right)} + o \left(1\right) \\[0.1cm]
\leq& \E{\mu \equiv 0,\lambda\equiv\sigma_0}{\left| \frac{1}{l_n} \sum\limits_{I_n \in \mathcal A_n} L_{n}\left(\Delta_n,I_n,\kappa_n,Y\right)-1\right|} + o \left(1\right),
\end{align*}
as $n \to \infty$. 
\end{proof}

\begin{proof}[Proof of Theorem \ref{thm:characterization}]
By Lemma \ref{lem:estimate_1} we have to prove that the right-hand side of \eqref{eq:estimate_1} tends to $0$ under the given conditions. 

Suppose that 
\begin{equation}\label{eq:aux_new}
\lim\limits_{n \to \infty} \E{\mu \equiv 0,\lambda\equiv\sigma_0}{1_{\left\{\left|L_{n}\left(\Delta_n,I_n,\kappa_n,Y\right) -1\right| \geq \xi l_n\right\}}\left|L_{n}\left(\Delta_n,I_n,\kappa_n,Y\right) -1\right|} = 0
\end{equation}
for any $\xi >0$. Then the weak law of large numbers for triangular arrays (the condition \eqref{eq:aux_new} implies the conditions imposed in \citep[Ex. 2.2]{s00}) is applicable (as $\left|I_n\right| \searrow 0$ and hence $l_n \nearrow \infty$), i.e. $m^{-1}\sum_{i=1}^m Z_i \to 1$ in probability with $m = l_n$ and $Z_i = L_{n}\left(\Delta_n,\left[\left(i-1\right)\left|I_n\right|,i\left|I_n\right|\right],\kappa_n,Y\right)$. As $Z_i >0$ we have 
\[
\E{}{\left|\frac1m\sum\limits_{i=1}^m Z_i\right|} = \frac1m\sum\limits_{i=1}^m\E{}{ \left|Z_i\right|} = \frac1m\sum_{i=1}^m\E{}{Z_i} = 1\qquad\text{ for all }m \in \mathbb N,
\]
which then implies $m^{-1}\sum_{i=1}^m Z_i \to 1$ in expectation, i.e. \eqref{eq:estimate_1} tends to $0$.

The only thing left to prove is that \eqref{eq:aux_new} holds true under the imposed conditions. To show this we will use the moments of $L_n = L_{n}\left(\Delta_n,I_n,\kappa_n,Y\right)$ under the null hypothesis. Note that under the null hypothesis $\mu \equiv 0$ it holds $Y_i \stackrel{\text{i.i.d.}}{\sim} \mathcal N \left(0,\sigma_0^2\right)$ and thus due to independence for $\eta>0$
\begin{align}
\E{\mu \equiv 0,\lambda\equiv\sigma_0}{L_n^\eta} =& \left(1+\kappa_n^2\right)^{-\eta\frac{n\left|I_n\right|}{2}} \E{\mu \equiv 0,\lambda\equiv\sigma_0}{\exp\left(\eta\sum\limits_{i: \frac{i}{n} \in A_{n,j}} \frac{\kappa_n^2Y_i^2 + 2Y_i \Delta_n - \Delta_n^2}{2 \left(1+\kappa_n^2\right) \sigma_0^2}\right)} \nonumber\\[0.1cm]
=&\left(1+\kappa_n^2\right)^{-\eta\frac{n\left|I_n\right|}{2}} \E{\mu \equiv 0,\lambda\equiv\sigma_0}{\prod\limits_{i: \frac{i}{n} \in A_{n,j}} \exp\left(\eta\frac{\kappa_n^2Y_i^2 + 2Y_i \Delta_n - \Delta_n^2}{2\left(1+\kappa_n^2\right) \sigma_0^2}\right)} \nonumber\\[0.1cm]
=& \left(1+\kappa_n^2\right)^{-\eta\frac{n\left|I_n\right|}{2}} \prod\limits_{i: \frac{i}{n} \in A_{n,j}} \E{}{\exp\left(\eta\frac{\kappa_n^2\sigma_0^2X^2 + 2\sigma_0X \Delta_n - \Delta_n^2}{2 \left(1+\kappa_n^2\right) \sigma_0^2}\right)} \nonumber\\[0.1cm]
=& \left(1+\kappa_n^2\right)^{-\eta\frac{n\left|I_n\right|}{2}} \left(\E{}{\exp\left(\eta\frac{\kappa_n^2\sigma_0^2X^2 + 2\sigma_0X \Delta_n - \Delta_n^2}{2\left(1+\kappa_n^2\right) \sigma_0^2}\right)} \right)^{n \left|I_n\right|},\label{eq:aux2}
\end{align}
where $X \sim \mathcal N \left(0,1\right)$. Thus we need to calculate the expectation of
\begin{equation}\label{eq:aux4}
\exp\left(\eta\frac{\kappa_n^2\sigma_0^2 X^2 + 2\sigma_0X \Delta_n - \Delta_n^2}{2\left(1+\kappa_n^2\right) \sigma_0^2}\right) = \exp\left(- \frac{ \eta\Delta_n^2}{2\sigma_0^2\kappa_n^2}\right) \cdot \exp\left(\eta s \left(X+\lambda\right)^2\right)
\end{equation}
where we abbreviated
\begin{align}
s &= \frac{\kappa_n^2}{2\left(1+\kappa_n^2\right)} \\[0.1cm]
\lambda &= \frac{\Delta_n}{ \sigma_0\kappa_n^2}.
\end{align}
The right-hand side in \eqref{eq:aux4} corresponds to the Laplace transform of a non-central $\chi^2$ distributed random variable given by
\begin{equation}\label{eq:laplace_transform}
\E{}{\exp\left(t \left(X + \lambda\right)^2\right)} = \frac{\exp\left(\frac{\lambda^2 t}{1-2t}\right)}{\sqrt{1-2t}}
\end{equation}
if $\lambda  \in \mathbb R$, $t <1/2$, otherwise the Laplace transform does not exist. Note that $t = \eta s <1/2$ if and only if $\eta < 1 + 1/\kappa_n^2)$. In this case the expectation is given by
\begin{align*}
\E{}{\exp\left(\eta\frac{\kappa_n^2\sigma_0^2X^2 + 2\sigma_0X \Delta_n - \Delta_n^2}{2\left(1+\kappa_n^2\right) \sigma_0^2}\right)} =& \exp\left(- \frac{\eta\Delta_n^2}{2\sigma_0^2\kappa_n^2}\right)\E{}{\exp\left(\eta s \left(X + \lambda\right)^2\right)} \\[0.1cm]
=& \frac{\sqrt{1+\kappa_n^2}}{\sqrt{1+\left(1 - \eta\right) \kappa_n^2}} \exp\left(\frac{\eta \left(\eta-1\right) \Delta_n^2}{2 \sigma_0^2 \left(1+\left(1 - \eta\right) \kappa_n^2\right)}\right). 
\end{align*}
By \eqref{eq:aux2} this implies
\begin{align*}
&\E{\mu \equiv 0,\lambda\equiv\sigma_0}{L_n^\eta} \\[0.1cm]
=& \left(1+\kappa_n^2\right)^{-\eta\frac{n\left|I_n\right|}{2}} \left(\frac{\sqrt{1+\kappa_n^2}}{\sqrt{1+\left(1 - \eta\right) \kappa_n^2}} \exp\left(\frac{\eta \left(\eta-1\right) \Delta_n^2}{2 \sigma_0^2 \left(1+\left(1 - \eta\right) \kappa_n^2\right)}\right)\right)^{n \left|I_n\right|} \\[0.1cm]
=& \frac{\left(1+\kappa_n^2\right)^{-\left(\eta-1\right)\frac{n\left|I_n\right|}{2}}}{\left(1+\left(1 - \eta\right) \kappa_n^2\right)^{\frac{n\left|I_n\right|}{2}}} \exp\left(\frac{\eta \left(\eta-1\right) n \left|I_n\right| \Delta_n^2}{2 \sigma_0^2 \left(1+\left(1 - \eta\right) \kappa_n^2\right)}\right).
\end{align*}
To prove \eqref{eq:aux_new} we now note for any $n$ such that $l_n \geq 1/\xi$ as $L_n>0$ we have
\begin{align}
&\E{\mu \equiv 0,\lambda\equiv\sigma_0}{1_{\left\{\left|L_n -1\right| \geq \xi l_n\right\}}\left|L_n -1\right|}\nonumber \\[0.1cm]
=& \int\limits_{\xi l_n}^\infty \Prob{\mu \equiv 0,\lambda\equiv\sigma_0}{L_n \geq x+1}\,\mathrm d x + \xi l_n \Prob{\mu \equiv 0,\lambda\equiv\sigma_0}{L_n \geq \xi l_n +1}\nonumber\\[0.1cm]
\leq& \int\limits_{\xi l_n}^\infty \Prob{\mu \equiv 0,\lambda\equiv\sigma_0}{L_n \geq x}\,\mathrm d x + \xi l_n \Prob{\mu \equiv 0,\lambda\equiv\sigma_0}{L_n \geq \xi l_n}\nonumber\\[0.1cm]
=& \E{\mu \equiv 0,\lambda\equiv\sigma_0}{1_{\left\{L_n \geq \xi l_n\right\}}L_n} \nonumber\\[0.1cm]
\leq& \E{\mu \equiv 0,\lambda\equiv\sigma_0}{\left(L_n\right)^{1+\delta}} \left(\xi l_n\right)^{-\delta} \nonumber\\[0.1cm]
=&\exp\left(\frac{n \left|I_n\right|\Delta_n^2}{2 \sigma_0^2} \frac{\left(1+\delta\right)\delta}{1-\delta\kappa_n^2} - \delta \frac{n\left|I_n\right|}{2} \log\left(1+\kappa_n^2\right) - \frac{n \left|I_n\right|}{2} \log\left(1- \delta\kappa_n^2\right) + \delta \log\left(\left|I_n\right|\right) - \delta \log\left(\xi\right) \right).
\label{eq:aux3}
\end{align}
Now, for \eqref{eq:aux3} to tend to $0$, the exponent must converge to $-\infty$. Consequently, the weak law of large numbers is applicable and finally \eqref{eq:estimate_1} tends to $0$ if \eqref{eq:cond} holds true.
\end{proof}

\begin{proof}[Proof of Theorem \ref{thm:LR_gen}]
Recall that the test statistic is given by
\[
T_n(Y)=\sup_{A_n\in\mathcal A_n} S(A_n),
\]
with the inner part
\[
S(A_n):= \frac{1}{\sigma_0^2} \sum\limits_{i: \frac{i}{n} \in I_n} \left(Y_i + \frac{\Delta_n}{\kappa_n^2} \right)^2
\]
We readily see that $S\left(A_n\right)$ obeys the following distributions
\begin{align*}
\text{Under }H_0: &\qquad S(A_n) \sim \chi^2_{n\left|I_n\right|} \left(\frac{n \left|I_n\right|\Delta_n^2}{\sigma_0^2\kappa_n^4}\right), \\[0.1cm]
\text{Under }H_1^n: &\qquad S(I_n) \sim \left(1+\kappa_n^2\right) \chi^2_{n\left|I_n\right|} \left(\left(1+\kappa_n^2\right)\frac{n\left|I_n\right|\Delta_n^2}{\sigma_0^2\kappa_n^4}\right).
\end{align*}
Here $I_n$ in the alternative denotes the true position of the jump. 

Using \eqref{eq:tail_upper} with $k = 1$, $b_1 = 1$, $d_1 = n \left|I_n\right|$ and $a_1^2 = \frac{n\left|I_n\right|}{\sigma_0^2 \kappa_n^4}$ we find
\begin{align*}
\Prob{H_0}{\sup_{I_n\in\mathcal A_n} S(I_n)>c_{\alpha,n}^*} \leq & \frac1{\left|I_n\right|} \Prob{H_0}{S(I_n)>c_{\alpha,n}^*} \\[0.1cm]
=&\frac1{\left|I_n\right|} \Prob{}{\chi^2_{n\left|I_n\right|} \left(\frac{n \left|I_n\right|\Delta_n^2}{\sigma_0^2\kappa_n^4}\right) > c_{\alpha,n}^*} \\[0.1cm]
\leq &\alpha.
\end{align*}

For the type II error denote
\[
y_{n,\alpha}=\left(1+\kappa_n^2\right)n\left|I_n\right| + \left(1+\kappa_n^2\right)^2 \frac{n\left|I_n\right|\Delta_n^2}{\sigma_0^2 \kappa_n ^4} - 2(1+\kappa_n^2)\sqrt{\left(n \left|I_n\right| + 2 \left(1+\kappa_n^2\right)\frac{n\left|I_n\right|\Delta_n^2}{\sigma_0^2\kappa_n^4}\right)\log\frac1\alpha}.
\]
Then we calculate using \eqref{eq:tail_lower} with $k = 1$, $b_1 = \left(1+\kappa_n^2\right)$, $d_1 = n \left|I_n\right|$ and $a^2 = \left(1+\kappa_n^2\right)\frac{n\left|I_n\right|\Delta_n^2}{\sigma_0^2 \kappa_n^4}$ that
\begin{align*}
%\Prob{H_1^n}{\sup\limits_{A_n\in\mathcal A_n} S(A_n)<y_{n,\alpha}}= & 
\sup\limits_{B_n \in \mathcal A_n} \Prob{B_n}{\sup\limits_{A_n\in\mathcal A_n} S(A_n)<y_{n,\alpha}} \leq & \sup\limits_{B_n \in \mathcal A_n} \inf\limits_{A_n\in\mathcal A_n} \Prob{B_n}{S(A_n)<y_{n,\alpha}} \\[0.1cm]
\leq &\sup\limits_{B_n \in \mathcal A_n} \Prob{B_n}{S(B_n)<y_{n,\alpha} }\\[0.1cm]
=& \Prob{} {\left(1+\kappa_n^2\right)\chi^2_{n\left|I_n\right|} \left(\left(1+\kappa_n^2\right)\frac{n\left|I_n\right|\Delta_n^2}{\sigma_0^2\kappa_n^4}\right) <y_{n,\alpha}} \\[0.1cm]
\leq &\alpha.
\end{align*}
Thus the claim is proven if $y_{n,\alpha}\geq c_{n,\alpha}^*$, i.e.
\begin{align*}
& \left(1+\kappa_n^2\right)n\left|I_n\right| + \left(1+\kappa_n^2\right)^2 \frac{n\left|I_n\right|\Delta_n^2}{\sigma_0^2 \kappa_n ^4} - 2(1+\kappa_n^2)\sqrt{\left(n \left|I_n\right| + 2 \left(1+\kappa_n^2\right)\frac{n\left|I_n\right|\Delta_n^2}{\sigma_0^2\kappa_n^4}\right)\log\left(\frac1\alpha\right)} \\[0.1cm]
\geq & n \left|I_n\right| + \frac{n \left|I_n\right|\Delta_n^2}{\sigma_0^2 \kappa_n^4} -2 \log\left(\alpha \left|I_n\right|\right) + 2 \sqrt{n \left|I_n\right|\left(1+2\frac{\Delta_n^2}{\sigma_0^2\kappa_n^4}\right)\log\left(\frac{1}{\alpha\left|I_n\right|}\right)}.
\end{align*}
But it can be easily seen by rearranging terms and multiplying by $\kappa_n^2$ that this is equivalent to \eqref{eq:cond_LR_test}. 
\end{proof}

\section{Proofs of Section \ref{sec:lower_bounds}}

If $\delta_n\kappa_n^2 \to 0$ and / or $\kappa_n \to 0$, Taylor's formula can be used to simplify some of the terms in \eqref{eq:cond}, which we will use in the following:
\begin{subequations}
\begin{align}
\frac{\left(1+\delta_n\right)\delta_n}{1-\delta_n\kappa_n^2} &=  \delta_n + \delta_n^2\left(1+\kappa_n^2\right) + \delta_n^3\kappa_n^2\left(1+\kappa_n^2\right) + \kappa_n^4\delta_n^4+\mathcal O \left(\kappa_n^6\delta_n^4\right), \label{eq:taylor1}\\[0.1cm]
-\log\left(1- \delta_n\kappa_n^2\right) &= \delta_n\kappa_n^2 + \frac{\delta_n^2\kappa_n^4}{2} + \mathcal O \left(\kappa_n^6\delta_n^3\right),\label{eq:taylor3}\\[0.1cm]
- \delta_n \log\left(1+\kappa_n^2\right) &= -\delta_n \kappa_n^2 +\frac{\delta_n\kappa_n^4}{2} +  \mathcal O \left(\kappa_n^6\delta_n \right). \label{eq:taylor2}
\end{align}
\end{subequations}

\begin{proof}[Proof of Theorem \ref{thm:db_equilibrium}]
Note that $\delta_n < \frac{1}{\kappa_n^2}$ is satisfied for free here as the right-hand side diverges or is constant and $\delta_n \to 0$. Furthermore in both cases the assumption $n \left|I_n\right| \Delta_n = \mathcal O \left(\log\left(\left|I_n\right|\right)\right)$ is satisfied as well. Thus we can apply Theorem \ref{thm:characterization}.
\begin{enumerate}
\item Under the assumptions of this part of the theorem, we may use \eqref{eq:taylor1}--\eqref{eq:taylor2} and insert $\kappa_n^2 = c \Delta_n / \sigma_0 \left(1+o\left(\eps_n\right)\right)$ to find that \eqref{eq:cond} is satisfied if
\begin{equation}\label{eq:cond_b_equiv}
\delta_n \left[\frac{n \left|I_n\right| \Delta_n^2}{2\sigma_0^2}\left(1+\frac{c^2}{2}\right) + \log\left(\left|I_n\right|\right) \right] + \delta_n^2 \frac{n \left|I_n\right| \Delta_n^2}{2\sigma_0^2} \left(1+\frac{c^2}{2}\right)  + o\left(\delta_n\eps_n n \left|I_n\right| \Delta_n^2\right)\to - \infty.
\end{equation}
Here we used that $\Delta_n = o \left(\eps_n\right)$. Then with $\delta_n := \sigma_0^{-1}\eps_n$ and using \eqref{eq:lower_bound_equilibrium1} we obtain 
\begin{align*}
\delta_n \left[\frac{n \left|I_n\right| \Delta_n^2}{2\sigma_0^2}\left(1+\frac{c^2}{2}\right) + \log\left(\left|I_n\right|\right) \right] &= \sigma_0^{-1}\eps_n \log\left(\left|I_n\right|\right)\left(1- \left(C-\eps_n\right)^2C^{-2}\right) \\[0.1cm]
&= 2\sigma_0^{-1}C^{-1}\eps_n^2 \log\left(\left|I_n\right|\right) - \sigma_0^{-1} C^{-2}\eps_n^3\log\left(\left|I_n\right|\right)
\end{align*}
and furthermore
\begin{align*}
& \delta_n^2 \frac{n \left|I_n\right| \Delta_n^2}{2\sigma_0^2} \left(1+\frac{c^2}{2}\right) + 2\sigma_0^{-1}C^{-1}\eps_n^2 \log\left(\left|I_n\right|\right) \\[0.1cm]
=& \sigma_0^{-1}\eps_n^2 \log\left(\left|I_n\right|\right) \left(-\sigma_0^{-1}\left(C-\eps_n\right)^2C^{-2}+ 2 C^{-1} \right) \\[0.1cm]
=& \left(\sigma_0^{-2}\left(\sqrt{2+c^2}-1\right)  \eps_n^2 + \mathcal O \left(\eps_n^3\right) \right) \log\left(\left|I_n\right|\right).
\end{align*}
It can be readily seen that $C_1 := \sqrt{2+c^2}-1\geq \sqrt{2}-1>0$ for all $c \geq 0$. Thus under \eqref{eq:lower_bound_equilibrium1} we have
\begin{align*}
&\delta_n \left[\frac{n \left|I_n\right| \Delta_n^2}{2\sigma_0^2}\left(1+\frac{c^2}{2}\right) + \log\left(\left|I_n\right|\right) \right] + \delta_n^2 \frac{n \left|I_n\right| \Delta_n^2}{2\sigma_0^2} \left(1+\frac{c^2}{2}\right) \\[0.1cm]
=&\left(C_1\eps_n^2 + o\left(\eps_n^2\right) \right) \log\left(\left|I_n\right|\right) \to -\infty
\end{align*}
as $n  \to \infty$. Thus \eqref{eq:cond_b_equiv} is satisfied.
\item Here only \eqref{eq:taylor1} and \eqref{eq:taylor3} can be applied. Thus \eqref{eq:cond} is satisfied if
\begin{equation}\label{eq:cond_b_equiv3}
\begin{aligned}
&\delta_n \left[n \left|I_n\right| \left(\frac{\Delta_n^2}{2\sigma_0^2} - \frac12\log\left(1+\kappa^2\right) + \frac{\kappa^2}{2} + o \left(\eps_n\right)\right) + \log\left(\left|I_n\right|\right) \right] \\[0.1cm]
+& \delta_n^2 n \left|I_n\right| \left[\frac{\Delta_n^2}{2\sigma_0^2}\left(1+\kappa^2\right)+  \frac{\kappa^4}{4} + o \left(\eps_n\right)\right] + \mathcal O \left(\delta_n^3 n \left|I_n\right|\right) \to - \infty
\end{aligned}
\end{equation}
Inserting the assumption $\Delta_n = \frac{\sigma^2}{c\sigma_0}\left(1+o\left(\eps_n\right)\right) = \frac{\sigma_0 \kappa^2}{c}\left(1+o\left(\eps_n\right)\right)$, \eqref{eq:lower_bound_equilibrium2} and $\delta_n := \eta \eps_n$ with $\eta>0$ to be chosen later we find
\begin{align*}
&\delta_n \left[n \left|I_n\right| \left(\frac{\Delta_n^2}{2\sigma_0^2} - \frac12\log\left(1+\kappa^2\right) + \frac{\kappa^2}{2} + o \left(\eps_n\right)\right) + \log\left(\left|I_n\right|\right) \right] \\[0.1cm]
=&\eta\eps_n \log\left(\left|I_n\right|\right) \left[1-\left(C-\eps_n\right)^2\left(\frac{\kappa^2}{2c^2}\left(\kappa^2 + c^2\right) - \frac12\log\left(1+\kappa^2\right)\right) + o \left(\eps_n\right)\right] \\[0.1cm]
=& 2\eta C^{-1} \eps_n^2 \log\left(\left|I_n\right|\right)\left(1 + o\left(1\right)\right)
\end{align*}
and
\begin{align*}
&\delta_n^2 n \left|I_n\right| \left[\frac{\Delta_n^2}{2\sigma_0^2}\left(1+\kappa^2\right)+ \frac{\kappa^4}{4} + o \left(\eps_n\right)\right] + 2\eta C^{-1} \eps_n^2 \log\left(\left|I_n\right|\right)\left(1 + o\left(1\right)\right) + \mathcal O \left(\delta_n^3 n \left|I_n\right|\right)\\[0.1cm]
=&  \eta^2\eps_n^2 \log\left(\left|I_n\right|\right)  \left( 2\eta^{-1}C^{-1} - \left(C-\eps_n\right)^2 \left(\frac{\kappa^4}{2c^2}\left(1+\kappa^2\right) + \frac{\kappa^4}{4}\right) + o\left(1\right)\right) \\[0.1cm]
=& \eta^2\eps_n^2 \log\left(\left|I_n\right|\right)  \left( 2\eta^{-1}C^{-1} - C^2\left(\frac{\kappa^4}{2c^2}\left(1+\kappa^2\right) + \frac{\kappa^4}{4}\right)  + o\left(1\right)\right) \to - \infty
\end{align*}
if we choose e.g. $\eta^{-1} = C^3 \left(\frac{\kappa^4}{2c^2}\left(1+\kappa^2\right) + \frac{\kappa^4}{4}\right)$. Thus \eqref{eq:cond} is satisfied. 
\end{enumerate}
\end{proof}

\begin{proof}[Proof of Theorem \ref{thm:db_signal}]
\begin{enumerate}
\item This follows directly from Theorem \ref{thm:db_equilibrium} with $c = c_n \to 0$.
\item Similarly as in the proof of Theorem \ref{thm:db_equilibrium} we have to show that \eqref{eq:cond_b_equiv3} is satisfied. Therefore note that with $\delta_n :=\eta\eps_n$ with $\eta>0$ to be chosen later we have
\begin{align*}
&\delta_n \left[n \left|I_n\right| \left(\frac{\Delta_n^2}{2\sigma_0^2} - \log\left(1+\kappa^2\right) + \frac{\kappa^2}{2} + o \left(\eps_n\right)\right) + \log\left(\left|I_n\right|\right) \right] \\[0.1cm]
=&\eta\eps_n \log\left(\left|I_n\right|\right) \left[1-\left(\sqrt{2}\sigma_0-\eps_n\right)^2 \left(\frac{1}{2\sigma_0^2} + \frac{1}{\Delta_n^2}\left(\frac{\kappa^2}{2}-\log\left(1+\kappa^2\right)\right)  + o \left(\eps_n\right)\right)\right] \\[0.1cm]
=& \sqrt{2}\eta\sigma_0^{-1} \eps_n^2 \log\left(\left|I_n\right|\right)\left(1+o\left(1\right)\right)
\end{align*}
by the assumption that $1/\Delta_n^2 =o \left(\eps_n\right)$. Furthermore
\begin{align*}
&\delta_n^2 n \left|I_n\right| \left[\frac{\Delta_n^2}{2\sigma_0^2}\left(1+\kappa^2\right)+ \frac{\kappa^4}{4}\right] + \sqrt{2}\eta\sigma_0^{-1} \eps_n^2 \log\left(\left|I_n\right|\right)\left(1+o\left(1\right)\right)\\[0.1cm]
=&\eta^2  \eps_n^2 \log\left(\left|I_n\right|\right) \left(\sqrt{2}\eta^{-1}\sigma_0^{-1} - \left(1+\kappa^2\right)+ o \left(1\right) \right) \to - \infty
\end{align*}
if we choose e.g. $\eta =\sigma_0^{-1}\left(1+\kappa^2\right)^{-1}$. This proves the claim.
\end{enumerate}
\end{proof}

\begin{proof}[Proof of Theorem \ref{thm:db_variance}]
Note again that $\delta_n < \frac{1}{\kappa_n^2}$ is satisfied for free here as the right-hand side diverges or is constant and $\delta_n \to 0$. Furthermore in both cases the assumption $n \left|I_n\right| \Delta_n = \mathcal O \left(\log\left(\left|I_n\right|\right)\right)$ is satisfied as well. Thus we can apply Theorem \ref{thm:characterization}.
\begin{enumerate}
\item Under the assumptions of this part of the theorem, we may again use \eqref{eq:taylor1}-\eqref{eq:taylor2} to find that \eqref{eq:cond} is satisfied if
\begin{equation}\label{eq:cond_b_equiv2}
\delta_n \left(\frac{ n \left|I_n\right| \Delta_n^2}{2 \sigma_0^2} + n \left|I_n\right|\frac{\kappa_n^4}{4}+\log\left(\left|I_n\right|\right)\right) +\delta_n^2 \left(\frac{n \left|I_n\right| \Delta_n^2}{2 \sigma_0^2} +n \left|I_n\right|\frac{\kappa_n ^4}{4} \right)  + o\left(\delta_n\eps_n\log\left(\left|I_n\right|\right)\right)\to -\infty
\end{equation}
Here we used that $\kappa_n^2 = o\left(\eps_n\right)$. By $\kappa_n^2 = \frac{\sigma_{n}^2}{\sigma_0^2} = \frac{\Delta_n}{\sigma_0} \frac{1}{\theta_n}$, we have under \eqref{eq:lower_bound_variance1} with $\delta_n := \sigma_0^{-1}\eps_n$ that
\begin{align*}
&\delta_n \left(\frac{ n \left|I_n\right| \Delta_n^2}{2 \sigma_0^2} + n \left|I_n\right|\frac{\kappa_n^4}{4}+\log\left(\left|I_n\right|\right)\right) \\[0.1cm]
=& \sigma_0^{-1}\eps_n \left(\frac{n \left|I_n\right|\Delta_n^2}{2\sigma_0^2} \left(1+\frac{1}{2\theta_n^2}\right) +\log\left(\left|I_n\right|\right)\right)\\[0.1cm]
=& \sigma_0^{-1}\eps_n\log\left(\left|I_n\right|\right) \left(1-\frac{\left(2\sigma_0-\eps_n\right)^2}{2\sigma_0^2}\left(1 + \theta_n^2\right)\right)\\[0.1cm]
=& -2\eps_n\log\left(\left|I_n\right|\right) \frac{\theta_n^2}{\sigma_0} + \left(\frac{2}{\sigma_0^2} \eps_n^2\log\left(\left|I_n\right|\right) - \frac{1}{2\sigma_0^3}\eps_n^3 \log\left(\left|I_n\right|\right)\right)\left(\frac12 + \theta_n^2\right) \\[0.1cm]
=& \eps_n^2\log\left(\left|I_n\right|\right) \left(\frac{2}{\sigma_0^2}  + o \left(1\right)\right),
\end{align*}
where we used $\theta_n^2  = o \left(\eps_n\right)$. Furthermore
\begin{align*}
&\delta_n^2 \left(\frac{n \left|I_n\right| \Delta_n^2}{2 \sigma_0^2} + n\left|I_n\right|\frac{\kappa_n^4}{4}\right) +\eps_n^2\log\left(\left|I_n\right|\right) \left(\frac{2}{\sigma_0^2}  + o \left(1\right)\right) + o\left(\delta_n\eps_n\log\left(\left|I_n\right|\right)\right)
\\[0.1cm]
=&\sigma_0^{-2} \eps_n^2\log\left(\left|I_n\right|\right) \left(2- \frac{\left(2\sigma_0-\eps_n\right)^2}{2\sigma_0^2}\left(\frac12 + \theta_n^2\right)+ o \left(1\right)\right)\\[0.1cm]
=& \eps_n^2\log \left(\left|I_n\right|\right) \left(1  + o \left(1\right)\right) \to - \infty.
\end{align*}
Thus \eqref{eq:cond_b_equiv2} and hence \eqref{eq:cond} is satisfied.
\item Similarly as in the proof of Theorem \ref{thm:db_equilibrium} we have to show that \eqref{eq:cond_b_equiv3} is satisfied. Therefore note that with $\delta_n := \eta\eps_n$ where $\eta>0$ will be chosen later we have
\begin{align*}
&\delta_n \left[n \left|I_n\right| \left(\frac{\Delta_n^2}{2\sigma_0^2} - \frac12\log\left(1+\kappa^2\right) + \frac{\kappa^2}{2}+ o\left(\eps_n\right)\right) + \log\left(\left|I_n\right|\right) \right] \\[0.1cm]
=&\eta\eps_n \log\left(\left|I_n\right|\right)\left[1+ \left(C-\eps_n\right)^2 \left( \frac12\log\left(1+\kappa^2\right) - \frac{\kappa^2}{2}\right)\right] + o\left(\eps_n^2 \log\left(\left|I_n\right|\right)\right)\\[0.1cm]
=&2C^{-1}\eta\eps_n^2 \log\left(\left|I_n\right|\right)\left(1+o\left(1\right)\right)
\end{align*}
by the assumption that $\Delta_n^2 = o \left(\eps_n\right)$. Furthermore
\begin{align*}
&\delta_n^2 n \left|I_n\right| \left[\frac{\Delta_n^2}{2\sigma_0^2}+ \frac{\kappa^4}{4}\right] +2C^{-1}\eta\eps_n^2 \log\left(\left|I_n\right|\right)\left(1+o\left(1\right)\right)\\[0.1cm]
=&\eta^2 \eps_n^2 \log\left(\left|I_n\right|\right) \left(2C^{-1}\eta^{-1}  - \left(C-\eps_n\right)^2\frac{\kappa^4}{4} + o \left(1\right) \right) \to - \infty
\end{align*}
if we choose e.g. $\eta^{-1} = C^3 \frac{\kappa^4}{4}$, which proves the claim.
\end{enumerate}
\end{proof}

\section{Proofs of Section \ref{sec:upper_bounds}}

\begin{proof}[Proof of Theorem \ref{thm:LR_equilibrium}]
Due to Theorem \ref{thm:LR_gen} we only have to show that under the assumptions of the theorem the condition \eqref{eq:cond_LR_test} is satisfied as $n \to \infty$. The given situation is $\kappa_n^2 = c \Delta_n / \sigma_0 \left(1+o\left(\eps_n\right)\right)$. Inserting this into \eqref{eq:cond_LR_test} we have to show that
\begin{align*}
&\left(c^2 + 2\right) \frac{n\left|I_n\right|\Delta_n^2}{\sigma_0^2} + c\frac{\Delta_n}{\sigma_0} \frac{n \left|I_n\right|\Delta_n^2}{\sigma_0^2}  \\[0.1cm]
\geq &\frac{2 c \Delta_n}{\sigma_0} \log\left(\frac1{\alpha\left|I_n\right|}\right) + 2 \sqrt{\frac{n\left|I_n\right|\Delta_n^2}{\sigma_0^2} \left(2+c^2 \right) \log\left(\frac{1}{\alpha \left|I_n\right|}\right)} \\[0.1cm]
&+2 \left(1+\frac{c\Delta_n}{\sigma_0}\right) \sqrt{\frac{n\left|I_n\right|\Delta_n^2}{\sigma_0} \left(c^2 + 2  + \frac{2c\Delta_n}{\sigma_0}\right) \log\left(\frac1\alpha\right)}
\end{align*}
as $n \to \infty$. Inserting \eqref{eq:condition_equilibrium} and dividing by $-\log\left(\left|I_n\right|\right)$ yields that 
\begin{align*}
&\left(c^2 + 2\right) \frac{\left(C+\eps_n\right)^2}{\sigma_0^2} + c\frac{\Delta_n}{\sigma_0} \frac{\left(C+\eps_n\right)^2}{\sigma_0^2}  \\[0.1cm]
\geq &\frac{2 c \Delta_n}{\sigma_0} \left(1+\frac{\log\alpha}{\log\left(|I_n|\right)}\right) + 2 \sqrt{\frac{\left(C+\eps_n\right)^2}{\sigma_0^2} \left(2+c^2 \right) \frac{\log\left(\alpha\right) + \log\left(\left|I_n\right|\right)}{\log\left(\left|I_n\right|\right)}} \\[0.1cm]
&+2 \left(1+\frac{c\Delta_n}{\sigma_0}\right) \sqrt{\left(c^2 + 2  + \frac{2c\Delta_n}{\sigma_0}\right)\frac{\left(C+\eps_n\right)^2}{\sigma_0^2} \frac{\log\left(\alpha\right)}{\log\left(\left|I_n\right|\right)}}
\end{align*}
is sufficient as $n \to \infty$. The definition of $C$ and some straightforward calculations using $\Delta_n = o \left(\eps_n\right)$ and $\log\alpha/\log (|I_n|) \to 0$ show that is enough to prove
\begin{align*}
4 + \frac{4 \sqrt{c^2+2}}{\sigma_0} \eps_n + \frac{c^2 + 2}{\sigma_0^2} \eps_n^2> 4 + \frac{2 \sqrt{2+c^2}}{\sigma_0} \eps_n + 8 \sqrt{\frac{\log\left(\alpha\right)}{\log\left(\left|I_n\right|\right)}} + \frac{4\sqrt{2+c^2}}{\sigma_0} \eps_n \sqrt{\frac{\log\left(\alpha\right)}{\log\left(\left|I_n\right|\right)}},
\end{align*}
which directly follows from $\eps_n \sqrt{-\log\left(\left|I_n\right|\right)} \to \infty$. 
\end{proof}

\begin{proof}[Proof of Theorem \ref{thm:LR_homogeneous}]
\begin{enumerate}
\item This follows directly from Theorem \ref{thm:db_equilibrium} with $c = c_n \to 0$.
\item Inserting the assumptions and dropping all lower-order contributions, we find with $\kappa:= \sigma  / \sigma_0$ that for \eqref{eq:cond_LR_test} the following is sufficient as $n \to \infty$:
\begin{align*}
\left(\kappa^2 + 2\right) \frac{n \left|I_n\right|\Delta_n^2}{\sigma_0^2} \geq &2 \kappa^2 \log \left(\frac{1}{\left|I_n\right|}\right) + 2\sqrt{2}\frac{\sqrt{n \left|I_n\right|}\Delta_n}{\sigma_0} \sqrt{\log \left(\frac{1}{\left|I_n\right|}\right)} \\[0.1cm]
&+2 \sqrt{2}\left(1+\kappa\right)^{3/2}\frac{\sqrt{n \left|I_n\right|}\Delta_n}{\sigma_0} \sqrt{\log\left(\frac1\alpha\right)}.
\end{align*}
Inserting \eqref{eq:condition_homogeneous} and dividing by $-\log\left(\left|I_n\right|\right)$, we see that it is enough to prove
\[
\left(\kappa^2 + 2\right) \frac{\left(\sqrt{2}\sigma_0 + \eps_n\right)^2}{\sigma_0^2} \geq 2 \kappa^2 + 2\sqrt{2}\frac{\left(\sqrt{2}\sigma_0 + \eps_n\right)}{\sigma_0}  + 2 \sqrt{2}\left(1+\kappa\right)^{3/2}\frac{\left(\sqrt{2}\sigma_0 + \eps_n\right)}{\sigma_0} \sqrt{\frac{\log\left(\frac1\alpha\right)}{\log\left(\left|I_n\right|\right)}}.
\]
But this condition is true due to our requirements on $\eps_n$.
\end{enumerate}
\end{proof}

\begin{proof}[Proof of Theorem \ref{thm:LR_dominantvar}]
Due to Theorem \ref{thm:LR_gen} we only have to show that under the assumptions of the theorem the condition \eqref{eq:cond_LR_test} is satisfied as $n \to \infty$. The given situation is $\kappa_n^2 = \Delta_n / (\sigma_0\theta_n)\left(1+o\left(\eps_n\right)\right)$. Inserting this into \eqref{eq:cond_LR_test} we have to show that
\begin{align*}
\left(2 + \frac{1}{\theta_n^2}\right) \frac{n\left|I_n\right|\Delta_n^2}{\sigma_0^2} + \frac{\Delta_n}{\sigma_0\theta_n} \frac{n \left|I_n\right|\Delta_n^2}{\sigma_0^2}  
\geq &\frac{2\Delta_n}{\sigma_0\theta_n} \log\left(\frac1{\left|I_n\right|}\right) + 2 \sqrt{\frac{n\left|I_n\right|\Delta_n^2}{\sigma_0^2} \left(2+\frac{1}{\theta_n^2}\right) \log\left(\frac{1}{\alpha \left|I_n\right|}\right)} \\[0.1cm]
&+2 \left(1+\frac{\Delta_n}{\sigma_0\theta_n}\right) \sqrt{\frac{n\left|I_n\right|\Delta_n^2}{\sigma_0} \left(2 + \frac{1}{\theta_n^2}  + \frac{2\Delta_n}{\sigma_0\theta_n}\right) \log\left(\frac1\alpha\right)}
\end{align*}
as $n \to \infty$. Inserting \eqref{eq:condition_dominantvar} and dividing by $-\log\left(\left|I_n\right|\right)$ yields that 
\begin{align*}
\frac{\left(2\sigma_0+\eps_n\right)^2}{\sigma_0^2}\left(2 \theta_n^2 +1\right)  + \frac{\left(2\sigma_0+\eps_n\right)^2}{\sigma_0^2}  \frac{\theta_n\Delta_n}{\sigma_0}\geq &2 \sqrt{\frac{\left(2\sigma_0+\eps_n\right)^2}{\sigma_0^2} \left(2\theta_n^2+1\right) \frac{\log\left(\alpha\right) + \log\left(\left|I_n\right|\right)}{\log\left(\left|I_n\right|\right)}} \\[0.1cm]
&+2 \sqrt{\frac{\left(2\sigma_0+\eps_n\right)^2}{\sigma_0^2}\left(2\theta_n^2 + 1\right) \frac{\log\left(\alpha\right)}{\log\left(\left|I_n\right|\right)}}
\end{align*}
is sufficient as $n \to \infty$ where we used that $\Delta_n / \theta_n \sim \sigma_{n}^2 = o\left(\eps_n\right)$. Some straightforward calculations and skipping all terms of order $o\left(\eps_n\right)$ shows that is enough to prove
\begin{align*}
4 + \frac{4}{\sigma_0} \eps_n > 4 \sqrt{\frac{\log\left(\alpha\right) + \log\left(\left|I_n\right|\right)}{\log\left(\left|I_n\right|\right)}}  +  \frac{2\eps_n}{\sigma_0} 4 \sqrt{\frac{\log\left(\alpha\right) + \log\left(\left|I_n\right|\right)}{\log\left(\left|I_n\right|\right)}} + 4 \sqrt{\frac{\log\left(\alpha\right)}{\log\left(\left|I_n\right|\right)}} + \frac{2\eps_n}{\sigma_0} \sqrt{\frac{\log\left(\alpha\right)}{\log\left(\left|I_n\right|\right)}}.
\end{align*}
Employing $\sqrt{a+b} \leq \sqrt{a}+\sqrt{b}$ we find the sufficient condition
\[
\frac{2}{\sigma_0} \eps_n > 8 \sqrt{\frac{\log\left(\alpha\right)}{\log\left(\left|I_n\right|\right)}}
\]
as $n \to \infty$, which directly follows from $\eps_n \sqrt{-\log\left(\left|I_n\right|\right)} \to \infty$. 
\end{proof}

\begin{proof}[Proof of Theorem \ref{thm:LR_adaptive}]
For $A_n \in \mathcal A_n$ let us abbreviate 
\begin{align*}
\bar Y_{A_n} &:= \left(n \left|I_n\right|\right)^{-1} \sum_{i: \frac{i}{n} \in A_{n}} Y_i, \\[0.1cm]
S_{A_n}^2 &:= \sum\limits_{i: \frac{i}{n} \in A_n} \left(Y_i - \bar Y_{A_n}\right)^2.
\end{align*}
In our model where the $Y_i$'s are independent Gaussian, it follows from Cochran's theorem that $S_{A_n}^2$ obeys a $\chi^2_{n\left|A_n\right|-1}$ distribution and $\bar Y_{A_n}$ an independent $\chi_1^2$ distribution. Now recall that the test statistic is given by
\[
T_n^*(Y)=\sup_{A_n\in\mathcal A_n} S(A_n),
\]
with the inner part
\[
S(A_n):= \frac{\kappa_n^2n\left|I_n\right|}{\sigma_0^2(\kappa_n^2+1)} S_{A_n}^2 + \frac{n\left|I_n\right|}{\sigma_0^2} \left(\bar Y_{A_n}\right)^2
\]
Using the above results, we readily see that $S\left(A_n\right)$ obeys the following distributions
\begin{align*}
\text{Under }H_0: &\qquad S(A_n) \sim\frac{\kappa_n^2}{\kappa_n^2+1} \chi^2_{n\left|I_n\right|-1} \left(0\right) + \chi_1^2 \left(0\right), \\[0.1cm]
\text{Under }H_1^n: &\qquad S(I_n) \sim \kappa_n^2 \chi^2_{n\left|I_n\right|-1} \left(0\right) + \left(\kappa_n^2+1\right)\chi_1^2 \left(\frac{n\left|I_n\right|\Delta_n^2}{\sigma_0^2\left(1+\kappa_n^2\right)}\right)
\end{align*}
Here $I_n$ in the alternative denotes the true position of the jump. 

For simplicity denote by $\chi_j^2$ a chi-squared random variable with $j$ degrees of freedom and by $\xi$ a standard normal variable independent of $\chi_j^2$. Now applying Lemma \eqref{eq:tail_upper} with $k = 2$, $b_1 = \frac{\kappa_n^2}{\kappa_n^2+1}$, $d_1= n \left|I_n\right|-1$, $b_2 = 1$, $d_2 = 1$ and $a_1 = a_2 = 0$ we get
\begin{align*}
\Prob{H_0}{\sup_{A_n\in\mathcal A_n} S(A_n)>c_{\alpha,n}^*} \leq & \frac1{\left|I_n\right|} \Prob{H_0}{S(A_n)>c_{\alpha,n}^*} \\[0.1cm]
=&\frac1{\left|I_n\right|} \Prob{}{ \frac{\kappa_n^2}{\kappa_n^2+1} \chi^2_{n\left|I_n\right|-1} + \xi^2>c_{\alpha,n}^*} \\[0.1cm]
\leq &\alpha.
\end{align*}
Let us turn to the type II error. We will apply \eqref{eq:tail_lower} with $k = 2$, $b_1 = \kappa_n^2$, $d_1= n \left|I_n\right|-1$, $b_2 = \kappa_n^2+1$, $d_2 = 1$, $a_1 = 0$ and $a_2^2 = \frac{n\left|I_n\right|\Delta_n^2}{\sigma_0^2\left(1+\kappa_n^2\right)}$. Denote
\[
y_{n,\alpha}=\kappa_n^2n\left|I_n\right| + 1 + \frac{\Delta_n^2 n\left|I_n\right|}{\sigma_0^2}-2\sqrt{\left[\kappa_n^4 n\left|I_n\right| + 2 \kappa_n^2 + 1 + \frac{2\left(1+\kappa_n^2\right)\Delta_n^2 n\left|I_n\right|}{\sigma_0^2}\right]\log\frac1\alpha}.
\]
We have
\begin{align*}
%\Prob{H_1^n}{\sup\limits_{A_n\in\mathcal A_n} S(A_n)<y_{n,\alpha}}= & 
\sup\limits_{B_n \in \mathcal A_n} \Prob{B_n}{\sup\limits_{A_n\in\mathcal A_n} S(A_n)<y_{n,\alpha}} \leq & \sup\limits_{B_n \in \mathcal A_n} \inf\limits_{A_n\in\mathcal A_n} \Prob{B_n}{S(A_n)<y_{n,\alpha}} \\[0.1cm]
\leq &\sup\limits_{B_n \in \mathcal A_n} \Prob{B_n}{S(B_n)<y_{n,\alpha} }\\[0.1cm]
=& \Prob{} {\kappa_n^2 \chi^2_{n\left|I_n\right|-1} \left(0\right) + \left(1+\kappa_n^2\right) \left(\xi + \sqrt{\frac{\Delta_n^2 n\left|I_n\right|}{\sigma_0^2(\kappa_n^2+1)}}\right)^2<y_{n,\alpha}} \\[0.1cm]
=& \Prob{} {\frac{\kappa_n^2}{\kappa_n^2+1} \chi^2_{n\left|I_n\right|-1} \left(0\right) + \left(\xi + \sqrt{\frac{\Delta_n^2 n\left|I_n\right|}{\sigma_0^2(\kappa_n^2+1)}}\right)^2<\frac{y_{n,\alpha}}{\kappa_n^2+1}} \\[0.1cm]
\leq &\alpha.
\end{align*}
Thus to find the detection boundary conditions we need to investigate the inequality $y_{n,\alpha}\ge c_{n,\alpha}^*$:
\begin{align*}
&\kappa_n^2n\left|I_n\right| + 1 + \frac{\Delta_n^2 n\left|I_n\right|}{\sigma_0^2}-2\sqrt{\left[\kappa_n^4 n\left|I_n\right| + 2 \kappa_n^2 + 1 + \frac{2\left(1+\kappa_n^2\right)\Delta_n^2 n\left|I_n\right|}{\sigma_0^2}\right]\log\frac1\alpha}\\[0.1cm]
\geq& \frac{\kappa_n^2 n\left|I_n\right| +1}{\kappa_n^2+1} +2\sqrt{\frac{\kappa_n^4n \left|I_n\right| + 2 \kappa_n^2 + 1}{\left(\kappa_n^2 + 1\right)^2}\log\left(\frac{1}{\alpha\left|I_n\right|}\right)} -2\log\left(\alpha\left|I_n\right|\right).
\end{align*}
First of all we can rewrite this inequality as follows,
\begin{align}
\frac{\kappa_n^4 n\left|I_n\right|}{\kappa_n^2+1}&+ \frac{\kappa_n^2}{\kappa_n^2+1}+\frac{\Delta_n^2 n\left|I_n\right|}{\sigma_0^2}-2\sqrt{\left[\kappa_n^4 n\left|I_n\right| + 2 \kappa_n^2 + 1 + \frac{2\left(1+\kappa_n^2\right)\Delta_n^2 n\left|I_n\right|}{\sigma_0^2}\right]\log\frac1\alpha}  \nonumber\\[0.1cm]
&\geq 2\sqrt{\frac{\kappa_n^4n \left|I_n\right| + 2 \kappa_n^2 + 1}{\left(\kappa_n^2 + 1\right)^2}\log\left(\frac{1}{\alpha\left|I_n\right|}\right)} -2\log\left(\alpha\left|I_n\right|\right).\label{ineq_cT}
\end{align}
Now we analyze it in the three regimes as usual:
\begin{itemize}
\item Dominant mean regime: Inserting $\kappa_n^2 / \Delta_n = o \left(\eps_n\right)$ and ignoring all $o \left(\eps_n\right)$-terms in the following, we find that it is sufficient to prove
\[
\frac{\Delta_n^2 n \left|I_n\right|}{\sigma_0^2} \geq 2 \log\left(\frac{1}{\left|I_n\right|}\right) + 2\sqrt{2} \sqrt{\frac{\Delta_n^2 n \left|I_n\right|}{\sigma_0^2}} \sqrt{\log\left(\frac{1}{\alpha}\right)}.
\]
Inserting \eqref{eq:adaptive_condition_homogeneous} and dividing by $-\log\left(\left|I_n\right|\right)$ we find that the above is the case if
\[
2 \sqrt{2} \sigma_0^{-1} \eps_n \geq 4 \sqrt{\frac{\log\left(\alpha\right)}{\log\left(\left|I_n\right|\right)}}
\]
which is true by assumption.
\item Equilibrium regime: Inserting $\kappa_n^2 / \Delta_n = o \left(\eps_n\right)$ and ignoring all $o \left(\eps_n\right)$-terms in the following, we find that it is sufficient to prove
\begin{multline*}
\frac{\Delta_n^2 n \left|I_n\right|}{\sigma_0^2} \left(1+c^2\right) \\[0.1cm] 
\geq 2 \log\left(\frac{1}{\left|I_n\right|}\right) + 2 \sqrt{\left(c^2 + 2\right)\frac{\Delta_n^2 n \left|I_n\right|}{\sigma_0^2}} \sqrt{\log\left(\frac{1}{\alpha}\right)} + 
2 c\sqrt{\frac{\Delta_n^2 n \left|I_n\right|}{\sigma_0^2}} \sqrt{\log\left(\frac{1}{\alpha\left|I_n\right|}\right)}.
\end{multline*}
Inserting \eqref{eq:adaptive_condition_equilibrium} and dividing by $-\log\left(\left|I_n\right|\right)$ we find that the above is the case if
\begin{multline*}
\frac{1+c^2}{\sigma_0^2} C^2 + 2 \frac{1+c^2}{\sigma_0^2} C \eps_n 
+\frac{(1+c^2)\eps_n^2}{\sigma_0^2}\\
\geq 2 + 2 C \frac{c}{\sigma_0} + \frac{2c}{\sigma_0} \eps_n + \frac{2(C+\eps_n)}{\sigma_0}\left(\sqrt{c^2 + 2}+c \right) \sqrt{\frac{\log\left(\alpha\right)}{\log\left(\left|I_n\right|\right)}}
\end{multline*}
which is true by the definition of $C$ and our assumption on $\eps_n$.
\item Dominant variance regime: Inserting $\kappa_n^2 = \Delta_n/ (\sigma_0 \theta_n)$ and ignoring all $o \left(\eps_n\right)$-terms in the following, we find that it is sufficient to prove
\begin{multline*}
\frac{\Delta_n^2 n \left|I_n\right|}{\sigma_0^2} \left(1+\frac{1}{\theta_n^2}\right) \\[0.1cm]
\geq 2 \log\left(\frac{1}{\left|I_n\right|}\right) + 2 \sqrt{\left(2 + \frac{1}{\theta_n^2}\right)\frac{\Delta_n^2 n \left|I_n\right|}{\sigma_0^2}} \sqrt{\log\left(\frac{1}{\alpha}\right)} + \frac{2}{\theta_n}\sqrt{\frac{\Delta_n^2 n \left|I_n\right|}{\sigma_0^2}} \sqrt{\log\left(\frac{1}{\alpha\left|I_n\right|}\right)}.
\end{multline*}
Let $C = 1+\sqrt{3}$. Inserting \eqref{eq:adaptive_condition_dominantvar}, dividing by $-\log\left(\left|I_n\right|\right)$, and using $\theta_n^2 = o \left(\eps_n\right)$ we find that the above holds true if
\[
C^2 + 2 \frac{C}{\sigma_0^2} \eps_n \geq 2 + 2 C + \left(2 C + 2 C^2\right)\sqrt{\frac{\log\left(\alpha\right)}{\log\left(\left|I_n\right|\right)}}
\]
which is true by the definition of $C$ and our assumption on $\eps_n$.
\end{itemize}
\end{proof}

\small
\bibliography{HBR_ref}{}
\bibliographystyle{plain}
\end{document}